\documentclass[a4paper]{amsart}
\usepackage{amssymb,amsfonts}
\usepackage[all,arc]{xy}
\usepackage{enumerate}
\usepackage{mathrsfs}
\usepackage{tikz}
\usetikzlibrary{arrows,snakes,backgrounds}
\usepackage{xcolor}   
\usepackage{marginnote}
\usepackage{tikz-cd}
\usetikzlibrary{positioning}
\usepackage{hyperref}
\usepackage{MnSymbol} 
\usepackage{mathtools}

\hypersetup{
    colorlinks=true, 
    linktoc=all,     
    citecolor=black,
    filecolor=black,
    linkcolor=black, 
    urlcolor=black
}
\usepackage{cleveref}
\usetikzlibrary{arrows}

\newtheorem{thm}{Theorem}[section]
\newtheorem{cor}[thm]{Corollary}
\newtheorem{prop}[thm]{Proposition}
\newtheorem{lem}[thm]{Lemma}

\newtheorem{quest}[thm]{Question}

\newtheorem*{openproblem*}{Problem}
\newtheorem*{quest*}{Question}
\newtheorem*{problem*}{Problem}
\newtheorem*{claim*}{Claim}

\theoremstyle{definition}
\newtheorem{defn}[thm]{Definition}

\theoremstyle{remark}
\newtheorem{rem}[thm]{Remark}

\newcommand{\bC}{\mathbf{C}}

\newcommand{\bN}{\mathbf{N}}

\newcommand{\bR}{\mathbf{R}}

\newcommand{\bZ}{\mathbf{Z}}

\newcommand\lra{\longrightarrow}

\newcommand\PL{\mathrm{PL}}
\newcommand\Diff{\mathrm{Diff}}
\newcommand\Homeo{\mathrm{Homeo}}

\newcommand\BdrDiff{\mathrm{BDiff}^{r}_\circ} 

\newcommand{\BH}{\mathrm{B}\text{\textnormal{Homeo}}}

\newcommand{\lin}{\ell^\infty}

\DeclareMathOperator{\FCh}{FCh}
\DeclareMathOperator{\RFCh}{RFCh}

\newcommand{\topo}{\tau}
\newcommand{\se}{\subseteq}
\newcommand{\rel}{\mathrm{near\ }}
\newcommand{\GV}{\mathscr{GV}}
\newcommand{\Euler}{\mathscr{E}}
\newcommand{\Pont}{\mathscr{P}}

\makeatother
\DeclareMathAlphabet{\mathpzc}{OT1}{pzc}{m}{it}

\addtocontents{toc}{\protect\setcounter{tocdepth}{1}}
\title[Bounded cohomology of homeomorphism and diffeomorphism groups]{Bounded and unbounded cohomology\\ of homeomorphism and diffeomorphism groups}
\author{Nicolas Monod}
\address{\'Ecole Polytechnique Fédérale de Lausanne (EPFL)\\
CH–1015 Lausanne,
Switzerland}
\email{nicolas.monod@epfl.ch}
\author{Sam Nariman}
\address{Department of Mathematics\\
  Purdue University\\
150 N. University Street\\
West Lafayette, IN 47907-2067\\
}
\email{snariman@purdue.edu}

\begin{document}
\begin{abstract}
We determine the bounded cohomology of the group of homeomorphisms of certain low-dimensional manifolds. In particular, for the group of orientation-preserving homeomorphisms of the circle and of the closed 2-disc, it is isomorphic to the polynomial ring generated by the bounded Euler class. These seem to be the first examples of groups for which the entire bounded cohomology can be described without being trivial.

We further prove that the $C^r$-diffeomorphisms groups of the circle and of the closed 2-disc have the same bounded cohomology as their homeomorphism groups, so that both differ from the ordinary cohomology of $C^r$-diffeomorphisms when $r>1$.

Finally, we determine the low-dimensional bounded cohomology of homeo- and diffeomorphism of the spheres $S^n$ and of certain 3-manifolds. In particular, we answer a question of Ghys by showing that the Euler class in $H^4(\Homeo_\circ(S^3))$ is unbounded.
\end{abstract}

\maketitle
\section{Introduction}
Bounded cohomology $H_b^\bullet$  has many motivations and applications; a few examples in geometry, topology and dynamics are~\cite{Gromov}, \cite{Ghys84}, \cite{Dupont}, \cite {Entov-Polterovich}, \cite{Bucher07}, \cite{Monod-Shalom2}. If a classical invariant is bounded, then determining a representative in bounded cohomology can refine the invariant considerably, see e.g.~\cite{Ghys_EM}, \cite{Monod-Shalom1}. Contrariwise, showing that a class in not bounded is a restriction by itself, see e.g.~\cite[Q.~F.1]{MR1271828}, \cite{calegari2004circular}.

Unfortunately, bounded cohomology remains largely elusive: as stated in \cite[\S1]{MonodICM}, there was no group for which all of $H_b^\bullet$ was known, unless it is trivial. Recent advances brought more examples where bounded cohomology is trivial and new examples where it is pathologically huge~\cite{Loeh_note17}, \cite{FFLM_2106.13567}.

Missing were examples where we can actually describe non-trivially, and understand, the entire bounded cohomology, especially for groups of relevance in geometry, topology or dynamics. \itshape This article is a first step in that direction.\upshape

Since our results concern groups of homeomorphisms and diffeomorphisms of manifolds, they can be compared to the very rich supply of results on the ordinary cohomology of these groups and how the latter relates to the underlying manifold. Interestingly, some of our results exhibit a close similarity to the classical case, while others stand in strong contrast. It turns out that this behaviour depends on the category: topological or smooth.

\subsection{The circle \texorpdfstring{$S^1$}{S1}}
Our first results regard the circle and are much simpler to prove than for the disc; they constitute therefore an excellent warm-up for the techniques introduced in this paper.

The ordinary cohomology of the group $\Homeo_\circ(S^1)$ of orientation-preserving homeomorphisms of the circle is known to be a polynomial ring generated by the Euler class $\Euler\in H^2$ for flat $S^1$-bundles:
\begin{equation}\label{eq:H:S1}
H^\bullet(\Homeo_\circ(S^1))\cong \bR[\Euler]
\end{equation}
(we take cohomology with coefficients in $\bR$). This can be deduced from a remarkable theorem that Thurston (\cite[Cor.~(b) of Thm.~5]{thurston1974foliations}) established for any closed manifold $M$. It states that that the natural map $B\Homeo(M)\to B\Homeo^\topo(M)$ induces a cohomology isomorphism, wherein $\Homeo^\topo$ denotes the \emph{topological} group endowed with the $C^0$-topology. (A group without superscript $\topo$ will always refer to the ``abstract'', i.e.\ discrete, group; see Section~\ref{sec:notation}.) Thurston's theorem applies to the neutral component $\Homeo_\circ$ as well. In fact, we always work with $\Homeo_\circ$, since this subgroup also determines the cohomology of $\Homeo$ by analysing the action of the group of components. In the case of the circle, we have $B\Homeo_\circ^\topo(S^1)\simeq \bC P^{\infty}$ and hence~\eqref{eq:H:S1} follows.

\medskip

Our first theorem is the analogue of~\eqref{eq:H:S1} for bounded cohomology. This is probably the first case where the entire bounded cohomology of a group can be determined without being either trivial or too pathologically large to describe. For the statement, recall that $\Euler$ is a \emph{bounded class} by the Milnor--Wood inequality (\cite{Milnor58}, \cite{Wood71}) and that it admits a \emph{unique} bounded representative ~\cite[ Cor.~2.11]{matsumoto1985bounded} in $H^2_b$, the \emph{bounded Euler class} $\Euler_b\in H^2_b$, because $\Homeo_\circ(S^1)$ is uniformly perfect (see~\cite[ Cor.~2.11]{matsumoto1985bounded} and~\cite[Thm.~2.3]{Eisenbud-Hirsch-Neumann}).

\begin{thm}\label{circle}
We have $H^\bullet_b(\Homeo_\circ(S^1))\cong \bR[\Euler_b]$, the polynomial ring generated by the bounded Euler class.
\end{thm}

Thus the comparison map $H^\bullet_b(\Homeo_\circ(S^1)) \to H^\bullet(\Homeo_\circ(S^1))$ is an isomorphism in every degree. Such a statement completely fails already for a surface $\Sigma_g$ of genus $g>0$. Indeed, a recent result of Bowden--Hensel--Webb (\cite{bowden2019quasi}) shows that $H^2_b(\Homeo_\circ(\Sigma_g))$ is infinite-dimensional; by contrast, the the results of Hamstrom~\cite{hamstrom1974homotopy} and Thurston~\cite{thurston1974foliations} imply that $H^2(\Homeo_\circ(\Sigma_g))$ vanishes. Quasimorphisms and low dimensional bounded cohomology of volume preserving diffeomorphisms and symplectomorphisms have also been extensively studied (see \cite{MR2104597, MR2276956, brandenbursky2021bounded} and references therein). Another example to which we shall return is a theorem due to Mann~\cite{mann2020unboundedness} stating that the map $H^2_b(\Homeo(M))\to H^2(\Homeo(M))$ has a nontrivial cokernel for certain Seifert fibered 3-manifolds $M$.

\medskip

For diffeomorphism groups, the usual group cohomology is related to deep open questions in foliation theory, in particular to the homotopy type of the Haefliger space (see \cite{MR1022692} and references therein). Let $\overline{B\Diff^r(M)}$ be the homotopy fiber  of the map
\[
\eta\colon B\Diff^r(M)\to B\Diff^{r,\topo}(M)
\]
that is induced by the identity homomorphism.  For $C^1$-diffeomorphisms, remarkably Tsuboi proved (see \cite{tsuboi1989foliated}) that $\eta$ is a homology isomorphism. In all regularities, Mather--Thurston's theorem (\cite[Thm.~5]{thurston1974foliations}, \cite{mather2011homology}) says that the space $\overline{B\Diff^r(M)}$ is homology isomorphic to the space of sections of a certain bundle over $M$. For $C^r$-diffeomorphisms when $r>1$, as a consequence of Mather--Thurston's theorem and a conjecture about the connectivity of the Haefliger space (\cite[Conjecture]{thurston1974foliations}, \cite[Section 6]{haefliger2010differential}),  it is expected that $\eta$ induces a homology isomorphism for degrees less than or equal to $\text{dim}(M)$. For regularities $r>1$ and a compact manifold $M$, it is known that there are non-trivial Godbillon--Vey classes in $H^{\text{dim}(M)+1}(\Diff^r(M);\bR)$  (see \cite{MR505649, MR3726710}) and in particular for the case of the circle, Thurston (\cite{thurston1972noncobordant}) proved that there is a surjective map
\[
H_2(\Diff_\circ^r(S^1);\bZ)\twoheadrightarrow \bR\oplus \bZ,
\]
where the map to $\bZ$ summand is induced by the Euler class and the map to $\bR$ is induced by the Godbillon--Vey class. It is conjectured (\cite{MR3726711}) that this map is in fact injective. Morita (\cite{morita1984nontriviality, MR877332, nariman2016powers}) generalised Thurston's theorem and showed that there exists a surjective map 
\[
H_{2n}(\Diff_\circ^r(S^1);\bZ)\twoheadrightarrow \bR\oplus \bZ,
\]
for all $n$ where the map to $\bZ$ summand is induced by the power of the Euler class. Although determining the group cohomology of $\Diff_\circ^r(S^1)$ for $r>1$ seems to be a very difficult open problem, the smooth group cohomology of the infinite dimensional Lie group $\Diff_\circ^r(S^1)$ satisfies the van Est isomorphism (\cite[Page 44]{MR660658}) and is completely determined as
\begin{equation*}
H^\bullet_{sm}(\Diff^r_\circ(S^1))\cong \bR[\Euler, \GV] \big/ (\Euler \cdot \GV = 0).
\end{equation*}
Similarly, the group cohomology of the piecewise linear homomeomorphisms $\PL_\circ(S^1)$ is much richer than the group cohomology of homeomorphisms of the circle. In fact, there is a discrete analogue of Godbillon--Vey classes (\cite{Ghys-Sergiescu}) that is non-trivial in $H^2(\PL_\circ(S^1))$. Combining Mather--Thurston's theorem for PL foliations and the work of Peter Greenberg (\cite[Cor.~1.12]{MR906823}), one can see that $H^\bullet(\PL_\circ(S^1))$ is much bigger than the ring that is generated by the Euler class.

In contrast to this classical picture, we prove that bounded cohomology does not distinguish between homeomorphisms, PL homeomorphisms and diffeomorphisms of $S^1$.

\begin{thm}\label{circle:diff}
For every $r\in\bN\cup\{\infty\}$, the inclusion $\Diff^r \to \Homeo$ induces an isomorphism in bounded cohomology:
\[
H^\bullet_b(\Diff^r_\circ(S^1))\ \cong\ H^\bullet_b(\Homeo_\circ(S^1))\cong \bR[\Euler_b].
\]
The same statement holds for the piecewise-linear category:
\[
H^\bullet_b(\PL_\circ(S^1))\ \cong\ H^\bullet_b(\Homeo_\circ(S^1))\cong \bR[\Euler_b].
\]
\end{thm}

As we shall see, the proof itself is rather robust with respect to changes of categories of circle transformations. For instance, it also holds for the \emph{countable} subgroup of $\PL_\circ(S^1)$ known as Thompson's group $T$. Therefore, the bounded cohomology of $T$ is also $\bR[\Euler_b]$; this contrasts with the ordinary cohomology determined by Ghys--Sergiescu~\cite{Ghys-Sergiescu} and provides a concrete example of a countable group whose bounded cohomology is entirely known. We note that the bounded cohomology of $T$ can also be determined by combining~\cite{Monod_wreath} with~\cite[\S6]{FFLM_binate_draft}.

\subsection{The closed disc \texorpdfstring{$D^2$}{D2}}
Some methods introduced for the proof of \Cref{circle} find their full use in higher dimensions. A rather more involved implementation of our strategy allows us to determine completely the bounded cohomology of $\Homeo_\circ(D^2)$.

\begin{thm}\label{thm:disk}
The restriction map $\Homeo_\circ(D^2)\to \Homeo_\circ(S^{1})$ induces an isomorphism in bounded cohomology.\\ In particular, the algebra $H^{\bullet}_b(\Homeo_\circ(D^2))$ is isomorphic to $\bR[\Euler_b]$.
\end{thm}

This recalls McDuff's theorem (\cite[Cor.~2.13]{mcduff1980homology}) stating that the restriction homomorphism $\Homeo_\circ(D^n)\to \Homeo_\circ(S^{n-1})$ induces an isomorphism in ordinary cohomology for all $n$.

It turns out that our proof is again robust when changing to the smooth category:

\begin{thm}\label{thm:disk:diff}
The statement of \Cref{thm:disk} also holds for $C^r$ diffeomorphisms, $r\in\bN\cup\{\infty\}$.\\
In particular, the algebra $H^{\bullet}_b(\Diff^r_\circ(D^2))$ is also isomorphic to $\bR[\Euler_b]$.
\end{thm}

 Bowden showed (\cite[Prop.~5.1]{MR2826937}) that in ordinary group cohomology not only the Euler class but also the Godbillon--Vey class pulls back nontrivially to the second group cohomology $H^2(\Diff^r_\circ(D^2);\bR)$ for $r>1$. Hence, the comparison map $H^2_b(\Diff^r_\circ(D^2))\to H^2(\Diff^r_\circ(D^2))$ is not an isomorphism for $r>1$.  Shigeyuki Morita mentioned to the authors the interesting fact that the Bott vanishing theorem (\cite{bott1970topological}), however, implies that $\Euler^4$ in  $H^8(\Diff^r_\circ(D^2))$ vanishes for $r>1$.

\begin{rem}
Let $\Diff^{r,vol}(D^2)$ be the volume preserving $C^r$-diffeomorphisms of $D^2$ for the standard volume form. It is known (\cite[Cor.~3.2]{MR1487633}) that for the restriction map $\Diff^{r,vol}(D^2)\to \Diff^r_\circ(S^1)$, the Euler class pulls back {\it trivially} to $H^2(\Diff^{r,vol}(D^2))$.
\end{rem}

Our general strategy is to construct a semi-simplicial set $X_\bullet$, depending on the regularity $r$, on which $G=\Diff^r_\circ(D^2)$ acts in such a way that we decompose $D^2$ into pieces, and each piece will have a stabiliser in $G$ that is easier to handle cohomologically. We then resolve $H^\bullet_b(G)$ in terms of $H^\bullet_b(X_\bullet)$, of $H^\bullet_b(X_\bullet/G)$ and of the stabilisers. In the case of $D^2$, the set $X_\bullet$ consists of configurations of neighbourhood germs of chords in the disc. The complexity of these configurations, specifically the infinite possibilities for mutual intersections of chords, has limited us to $n=2$.

For this strategy to actually simplify the problem, we need to show that the stabilisers are  \emph{boundedly acyclic}, meaning that their bounded cohomology vanishes in all positive degrees. We shall establish several such general bounded acyclicity results, both as tools for the above and for the case of $S^n$ below. The first non-trivial bounded acyclicity statement is the theorem of Matsumoto--Morita~\cite{matsumoto1985bounded} which states that the group $\Homeo_c(\bR^n)$ of \emph{compactly supported} homeomorphisms is boundedly acyclic. That result was a refinement of Mather's acyclicity in ordinary cohomology~\cite{MR0288777}. A tool for our study of homeomorphisms and diffeomorphisms of $S^n$ is the following generalisation of the Matsumoto--Morita theorem to the case of certain manifolds that can be ``displaced", and also to higher regularities.

\begin{thm}\label{thm:MR^n}
Let $M$ be any closed $C^r$-manifold and $n\geq 1$ and let $Z$ be $C^r$-diffeomorphic to $M\times \bR^n$. Then the groups $\Homeo_c(Z)$, $\Homeo_{c, \circ}(Z)$, $\Diff^r_c(Z)$ and $\Diff^r_{c, \circ}(Z)$ are boundedly acyclic for all $r\in\bN\cup\{\infty\}$.
\end{thm}

We note that earlier instances of the ``displacement'' technique, in degree two, include Kotschick~\cite{Kotschick2008} and Burago--Ivanov--Polterovich~\cite{burago2008conjugation} (notably their notion of portable manifold).

\subsection{Flat \texorpdfstring{$S^n$}{Sn}-bundles and Ghys's question}\label{sec:intro:Ghys}
Recall that the Milnor--Wood inequality (\cite{Milnor58}, \cite{Wood71}) for a $C^0$-flat $S^1$-bundle $ E\xrightarrow{p} \Sigma_g$ over a closed oriented surface $\Sigma_g$ of genus $g>0$ states
\[
|\langle \Euler(p), [\Sigma_g]\rangle|\leq 2g-2.
\]
Milnor proved this inequality when the flat circle bundle is linear meaning that the monodromy  group of the bundle lies in $\mathrm{PSL}_2(\bR)$ and Wood proved the case where the monodromy group lies in $\Homeo_\circ(S^1)$. In the bounded cohomology language, it says that the Euler class is a bounded class and has a Gromov norm equal to $\frac 12$. Ghys~\cite[\S~F.1]{MR1271828} asked the following question about a generalisation of the Milnor--Wood inequality to flat $S^3$-bundles:

\begin{quest}
Let $M^4$ be a compact orientable $4$-manifold and $ \pi_1(M) \xrightarrow{\rho} \Homeo_\circ(S^3)$ be a representation. Is it true that the Euler number of the associated $S^3$-bundle over $M$ is bounded by a number depending only on $M$? Is $\Euler\in H^4(\Homeo_\circ(S^3))$ a bounded class?
\end{quest}

One can define the Euler class for oriented $S^3$-bundles using classifying spaces as follows. Note that any oriented (not necessarily flat) $S^3$-bundle over the $4$-manifold $M$ up to isomorphism corresponds to a map $M\to B\Homeo_\circ(S^3)^\topo$ which is well-defined up to homotopy. By Smale's conjecture in the topological category (\cite{ hatcher1983proof}), we know that $\Homeo_\circ(S^3)^\topo\simeq \mathrm{SO}(4)$. Hence, there is a universal Euler class $\Euler$ in $H^4(B\mathrm{SO}(4))\cong H^4(B\Homeo_\circ(S^3)^\topo)$ and in fact $H^4(B\mathrm{SO}(4))\cong \bR^2$ which is generated by the universal first Pontryagin class $\Pont_1$ and $\Euler$ (see also \cite[Remark 3.3]{nariman2016powers}). Now as a consequence of Thurston's theorem (\cite[Cor.~(b) of Thm.~5]{thurston1974foliations}), we know the natural map
\[
H^4(B\Homeo_\circ(S^3)^\topo)\to H^4(\Homeo_\circ(S^3)),
\]
is an isomorphism. Therefore, the pull back of the universal $\Pont_1$ and $\Euler$  on the universal $C^0$-flat $S^3$-bundle are nontrivial. By the abuse of notation, we denote the pull back of these classes to flat bundles by the same notation. So $H^4(\Homeo_\circ(S^3))$  is also generated by $\Pont_1$ and $\Euler$. 

Back to Ghys's question above, proving the boundedness of $\Euler$ is the approach suggested by  the Gromov--Hirzebruch proportionality principle. This boundedness would generalize the Milnor--Wood inequality and be compatible with the Hirsch--Thurston theorem~\cite{MR370615} stating that the Euler class vanishes when $\pi_1(M)$ is amenable. Moreover, for \emph{linear} sphere bundles, the boundedness does indeed hold, as proved by Sullivan~\cite{Sullivan76} and Smillie~\cite{Smillie_unpublished} (see also~\cite{Ivanov-Turaev} for this boundedness, and~\cite{Bucher-Monod} for the exact bound).

\medskip


In \Cref{sec:sphere}, we use the higher dimensional version of the semi-simplicial set we used for the circle case to answer Ghys's question in the negative:

\begin{thm}\label{thm:Q:Ghys}
We have $H^{4}_b(\Homeo_\circ(S^3))=0$ and $H^{4}_b(\Diff^r_\circ(S^3))=0$ for all $r\neq 4$. In particular, the Euler class and the first Pontryagin class $\Pont_1$ in $H^4(\Homeo_\circ(S^3))$ are unbounded.
\end{thm}
\begin{rem} Our proof of the unboundedness of the Euler class for oriented $C^0$ flat $S^3$-bundles is not constructive. In particular, it would be very interesting to construct explicit families of flat $S^3$-bundles over a given 4-manifold with unbounded Euler number.
\end{rem}
As we shall see, it is not hard to use the same semi-simplicial set to prove:

\begin{thm}\label{thm:sphere:intro}
We have $H^{2}_b(\Homeo_\circ(S^n))=0$ and $H^{3}_b(\Homeo_\circ(S^n))=0$ for all~$n>1$.

The same holds for $\Diff^r_\circ(S^n)$ with $r\in\bN\cup\{\infty\}$.
\end{thm}

But to get the calculation up to $H^4_b$, we will use the homotopy type of the group of diffeomorphisms of $3$-dimensional pair of pants and also uniform perfectness of such groups in dimension $3$. To approach the same calculation for $S^2$, the main obstacle is that uniform perfectness is likely to fail in that context, in view of the work of Bowden--Hensel--Webb~\cite{bowden2019quasi}. For higher dimensional spheres, the main obstacle is the homotopy type of group of diffeomorphisms of  higher dimensional pair of pants.

In a first draft of this paper we established \Cref{thm:sphere:intro} for $n=2,3$ only, due to a restriction on stabilisers appearing for the semi-simplicial sets. Thanks to a result of Fournier-Facio and Lodha~\cite{FFL_2111_.07931}, this restriction is lifted. Likewise, we did not cover the case of diffeomorphisms. That generalisation was made possible by the recent note~\cite{Monod_wreath} (which also lifts the restriction on $n$).

\bigskip

\subsection*{Acknowledgements}
Upon becoming aware of our work, Fournier-Facio, L{\"o}h and Moraschini have kindly shown us a draft of their work~\cite{FFLM_binate_draft}. We are grateful for their comments. Although the two lines of investigation have essentially no overlap, their work gives an alternative proof of the bounded acyclicity of some of the stabilisers appearing within some proofs below. We are grateful to Jonathan Bowden for his comments on the group of homeomorphisms and diffeomorphisms of $3$-dimensional pair of pants. We  thank Shigeyuki Morita for his comment about the powers of the Euler class for flat $D^2$-bundles and also we thank Mehdi Yazdi for his careful reading and critical comments on the earlier draft of this paper. We are grateful to the anonymous referee for many comments which improved our exposition.

S.N. was partially supported by NSF DMS-2113828 and Simons Foundation (855209, SN).

\section{Notation and background results}\label{sec:notation}
We write $D^n \se \bR^n$ for the closed $n$-dimensional disc (ball), $B^n = \mathrm{int}(D^n)$ for the open one, and $S^{n-1} = \partial D^n$ for the $(n-1)$-sphere.

Notations such as $\Homeo$, $\Diff^r$ and their variants shall refer to the corresponding \emph{groups}, without any additional topological structure. If we drop the order of regularity $r$ for diffeomorphisms, we mean smooth diffeomorphisms. If we endow them with a topology, in this case the $C^0$ or $C^r$-topology, then we shall write $\Homeo^\topo$, $\Diff^{r, \topo}$ etc.\ for the resulting topological groups. We warn the reader that several authors adopt the opposite convention, where the groups are topologized by default and a superscript $\delta$ is added to indicate that e.g.\ $\Homeo^\delta$ is stripped of its topology, or endowed with the discrete topology.

(The present convention, arguably more pedantically precise, is convenient here since the objects of study are ``abstract'' groups of homeomorphism; it leads occasionally to an unusual notation such as $\mathrm{SO}(3)^\topo$ for the \emph{compact} group of rotations.)

Given a subset $Y\se X$ of a topological space $X$, we denote by $\Homeo(X;Y)$ the subgroup of $\Homeo(X)$ consisting of those homeomorphisms that fix $Y$ pointwise. We denote by $\Homeo(X;\rel Y)$ the subgroup of $\Homeo(X;Y)$ of all elements that fix pointwise a neighbourhood of $Y$; this neighbourhood depends a priori on the element.

If no coefficients are explicitly indicated, then cohomology, bounded cohomology and any function spaces are always understood with coefficients in $\bR$ viewed as a trivial module.

Given a semi-simplicial set $X_\bullet$, the \textbf{bounded cohomology} $H_b^\bullet(X_\bullet)$ of $X_\bullet$ refers to the cohomology of the associated complex of spaces of bounded functions
\[
0 \lra \lin(X_0) \lra \lin(X_1) \lra \lin(X_2) \lra \cdots
\]
Although $H_b ^n$ is still poorly understood for general $n$, we note that $X_\bullet$ is connected if and only if $H_b^0(X_\bullet)$ has dimension one. The semi-simplicial set is called \textbf{boundedly acyclic} if $H_b^n(X_\bullet)$ vanishes for all $n>0$.

The bounded cohomology $H_b^\bullet(G)$ of a group $G$ and the corresponding notion of bounded acyclicity are obtained by taking $X_\bullet$ to be the Milnor join of $G$. Equivalently, the \emph{homogeneous} resolution leads to an identification of $H_b^\bullet(G)$ with the cohomology of the complex of invariants
\[
0 \lra \lin(G)^G \lra \lin(G^2)^G \lra \lin(G^3)^G \lra \cdots
\]
where $G$ acts diagonally on $G^{\bullet+1}$ and where the differentials are given by the simplicial Alexander--Kolmogorov--Spanier face maps (alternating sums omitting variables).

\medskip

The following fundamental result of Matsumoto--Morita will be used extensively.

\begin{thm}[Matsumoto--Morita, Thm.~3.1 in~\cite{matsumoto1985bounded}]\label{thm:MM}
The group $\Homeo_c(\bR^n)$ of compactly supported homeomorphisms of $\bR^n$ (equivalently of $B^n$) is boundedly acyclic.\qed
\end{thm}

That result has recently been extended by various authors to other and wider settings. First, to all \emph{mitotic} groups~\cite{Loeh_note17}. Then, at the same time as we wrote the first version of the present article, to \emph{binate} groups~\cite{FFLM_binate_draft}. After this, a criterion was established in~\cite{Monod_wreath} which covers the following situation including numerous homeomorphism or diffeomorphism groups:

\begin{thm}[Cor.~6 in~\cite{Monod_wreath}]\label{thm:wreath}
Let $G$ be a group acting faithfully on a set $Z$. Suppose that $Z$ contains a subset $Z_0$ and that $G$ contains an element $g\in G$ with the following properties:

\begin{enumerate}[(i)]
\item every finite subset of $G$ can be conjugated so that all its elements are supported in $Z_0$;
\item $g^p(Z_0)$ is disjoint from $Z_0$ for every integer $p\geq 1$.
\end{enumerate}

\noindent
Then $G$ is boundedly acyclic.\qed
\end{thm}

In general we have no reason to believe that an infinite product of boundedly acyclic groups is necessarily boundedly acyclic. However, the above sufficient condition is stable under products; indeed, it suffices to consider the action of the product on the corresponding disjoint union $Z$ of sets, and to define $Z_0$ as the disjoint union of the corresponding subsets. Hence, we record the following.

\begin{lem}\label{lem:inf:power}
Any product of groups satisfying the conditions of \Cref{thm:wreath} will satisfy them too, and hence be boundedly acyclic.\qed
\end{lem}

Next, we state a fundamental stability result for the vanishing of bounded cohomology (without any claim of originality); it implies in particular that bounded acyclicity is preserved under finite direct products. 

\begin{prop}\label{prop:ext:van}
Let $G$ be a group, $K\lhd G$ a normal subgroup and $N\in\bN\cup\{\infty\}$.

If $H^n_b(K)$ vanishes for all $0<n<N$, then there is an isomorphism $H^n_b(G)\cong H^n_b(G/K)$ for all $0\leq n<N$.
\end{prop}

\begin{proof}
This follows from a bounded version of the Lyndon--Hochschild--Serre spectral sequence, see e.g.~\cite[\S12]{MonodLNM}. For the reader's convenience, we expand on the necessary details:

Define $Q=G/K$ and consider the double complex
\[
\ell^\infty(G^{p+1} \times Q^{q+1})^G \ \cong \ \ell^\infty\Big(G^{p+1}, \ell^\infty( Q^{q+1})\Big)^G \ \cong \ \ell^\infty\Big(Q^{q+1}, \ell^\infty( G^{p+1})\Big)^G
\]
indexed by $p,q\geq 0$, with the two differentials being given by the usual homogeneous differential on the variables in $G$ and in $Q$ respectively. We are now in a very special case of the setting considered in~\cite[\S12]{MonodLNM}, the differences being as follows. (1)~The reference allows for non-trivial coefficients $F$, while here $F=\bR$. (2)~The reference considers locally compact groups acting on measure spaces $S,T$, while here all groups are discrete and simply act on themselves, so $S=G$ and $T=Q$. (3)~The reference assumes that the locally compact groups are second countable in order to avoid difficulties with measurability; this assumption is not needed here since we have no measurability questions.

We can therefore quote from~\cite[\S12]{MonodLNM} as follows. The double complex gives rise to two spectral sequences, the first of which collapses already in the first page and abuts to the bounded cohomology $H^\bullet_b(G)$ of $G$, see Prop.~12.2.1 in~\cite{MonodLNM}. The second spectral sequence yields the same limit up to isomorphism and Prop.~12.2.2(ii) in~\cite{MonodLNM} states that the second page of this sequence is
\[
E_2^{p,q} \cong H^p_b\big(Q, H^q_b(K)\big)
\]
as soon as $H^q_b(K)$ is Hausdorff. Since $H^0_b(K)\cong \bR$ is always Hausdorff, $E_2^{p,0}$ is isomorphic to $H^p_b(Q)$ and our vanishing assumption on $H^n_b(K)$ allows us to deduce that $E_2^{p,q}$ vanishes for $0<q<N$ and all $p$. These facts together imply the stated isomorphisms $H^n_b(G)\cong H^n_b(G/K)$ for all $0\leq n<N$.
\end{proof}

In more elaborate vanishing arguments, it will be crucial to have a more refined, quantitative, control of the vanishing. To that end, we introduce the notion of \emph{vanishing moduli} as follows.

\begin{defn}\label{defn:modulus}
Let $X_\bullet$ be a semi-simplicial set satisfying $H^n_b(X_\bullet)=0$ for some $n$. Then we define the $n$th \textbf{vanishing modulus} of $X_\bullet$ as
%
%
%
\[
\sup_{c\in C_1} \inf \Big\{ \|b\|_\infty : b\in \ell^\infty(X_{n-1}) \text{ with } d b = c \Big\},
\]
where
\[
C_1 = \Big\{ c\in \ell^\infty(X_{n}) \text{ with } d c = 0 \text{ and } \|c\|_\infty \leq 1 \Big\}.
\]
\end{defn}

In other words, \itshape the vanishing modulus controls the smallest norm of a cochain $b$ that witnesses that a given cocycle $c$ is a coboundary\upshape. This modulus is finite by the open mapping theorem applied to the coboundary map $d$. This quantity is dual to the ``uniform boundary condition'' constant considered by Matsumoto--Morita~\cite{matsumoto1985bounded}, although they consider it more generally for the trivial cycles within possibly non-acyclic normed complexes.

In particular, if $G$ is a group with vanishing $H^n_b(G)$, we can speak of the $n$th vanishing modulus of $G$. It can be computed either on the Milnor join of $G$ or on the homogeneous resolution since they give isometric cochain complexes.

A first occurrence of this notion is as follows. The direct factors of a boundedly acyclic product of groups are boundedly acyclic, but we will need the following uniformity statement which is a priori stronger for infinite products.

\begin{prop}\label{prop:power:mod}
Fix a positive integer $q$ and let $(G_i)_{i\in I}$ be any family of groups.

If the product $\prod_{i\in I} G_i$ is boundedly acyclic, then the $q$th vanishing modulus of all subproducts $\prod_{j\in J} G_j$ is bounded independently of $J\se I$.

In particular, the $q$th vanishing modulus of all $G_i$ is bounded independently of $i$.
\end{prop}

The proof uses the following elementary technical observation, which we isolate for later reference.

\begin{lem}\label{lem:mod:subgroup}
Let $G$ be a group and $G_1<G$ a subgroup such that $H^q_b(G_1)=0$ for some $q>0$. Then the $q$th vanishing modulus of $G_1$ coincides with the $q$th vanishing modulus of the semi-simplicial orbit set given by $(G^{p+1})/G_1$ in each dimension $p$.
\end{lem}

\begin{proof}[Proof of \Cref{lem:mod:subgroup}]
It is known that the bounded cohomology $H^\bullet_b(G_1)$ can be realised \emph{isometrically} on the homogeneous resolution
\[
\cdots \lra \lin(G^q)^{G_1} \lra \lin(G^{q+1})^{G_1} \lra \lin(G^{q+2})^{G_1} \lra \cdots
\]
see e.g.~\cite[7.4.10]{MonodLNM} for a more general statement. The proof given in this reference proceeds by exhibiting maps between this resolution and the homogeneous resolution for $G_1$ which are non-expanding already at the cocycle level, which implies the statement of the lemma.

We point out that in the case considered here, the technical tools used in~\cite[7.4.10]{MonodLNM} simply boil down to extending functions from $G_1^{q+1}$ to $G^{q+1}$ by using coset representatives.
\end{proof}

\begin{proof}[Proof of \Cref{prop:power:mod}]
We claim that for a boundedly acyclic product $G=G_1 \times G_2$ of two groups, the factor $G_1$ is boundedly acyclic with its $q$th vanishing modulus bounded by that of $G$. This claim implies the statement of the proposition by regrouping the factors in the (possibly infinite) product.

Recall that inflation refers to the morphism $H_b^\bullet(G_1) \to H_b^\bullet(G)$ induced by the projection $G\to G_1$ while restriction is the morphism $H_b^\bullet(G) \to H_b^\bullet(G_1)$ induced by the inclusion $G_1\to G$; thus their composition gives the identity on $H_b^\bullet(G_1)$. Moreover, the inclusion of $G_1$-invariants into $G$-invariants realizes the inflation map. Therefore, we can compute the norms of the relevant cocycles and coboundaries while realising the inflation-restriction morphisms by the following two inclusions of (differential complexes of) Banach spaces:
\[
\ell^\infty((G_1)^{q+1})^{G_1} = \ell^\infty((G/G_2)^{q+1})^{G} \se  \ell^\infty(G^{q+1})^{G} \se  \ell^\infty(G^{q+1})^{G_1}.
\]
The fact that we obtain the correct norms also in the right hand side follows from \Cref{lem:mod:subgroup} and hence the claim is established.
\end{proof}

\section{Bounded cohomology of semi-simplicial sets}
The following is a versatile method for constructing boundedly acyclic semi-simplicial sets.

\begin{defn}\label{defn:generic}
Let $X$ be a set and let $\perp$ be any binary relation on $X$. We say that the relation $\perp$ is \textbf{generic} if, given any finite set $F\se X$, there exists $x\in X$ with $x \perp y$ for all $y\in F$.

Furthermore, we define a semi-simplicial set $X^\perp_\bullet$ as follows: $X^\perp_n$ consists of all $(n+1)$-tuples $(x_0, \ldots, x_n)$ in $X^{n+1}$ such that $x_i \perp x_j$ holds for all $i < j$. Note that $X^\perp_n$ is non-empty when $\perp$ is generic (and $X\neq\varnothing$). The face maps are defined to be the usual simplex face maps.
\end{defn}

In general $\perp$ will be thought to represent a suitable transversality condition. A basic example is the relation $\neq$, which is generic as soon as $X$ is infinite. A non-symmetric example is the relation $<$ on $X=\bR$. We shall later see more interesting examples with tuples of germs of arcs.

We emphasize that even though $X^\perp_\bullet$ consists of pairwise-$\perp$ tuples, the definition of genericity must be verified for \emph{all} finite sets $F$, without assuming any relation between their elements. 

\begin{prop}\label{prop:generic}
If $\perp$ is generic, then $X^\perp_\bullet$ is boundedly acyclic (and connected).
\end{prop}

\begin{proof}
Given $F\se X$ finite, let $V_F = \{ x\in X : x \perp y \forall \,y\in F\}$. Note that $V_F \cap V_{F'} \supseteq V_{F \cup F'}$. Thus our assumption implies that the collection of all such $V_F$ is a proper filter base on $X$. Choose an ultrafilter $U$ on $X$ containing the filter generated by this base. For any \emph{bounded} function $f$ on $X$, we can therefore consider the ultralimit $\lim_{x\to U} f(x)$.

We define a map
\[
h_n \colon \lin(X^\perp_n) \lra  \lin(X^\perp_{n-1})
\]
by
\[
(h_n f)(x_1, \ldots, x_n) = \lim_{x\to U} f(x, x_1, \ldots, x_n).
\]
We note that $h_n$ is linear and bounded (of norm one); the crucial point is that it is well-defined, since the collection of $x$ for which $f(x, x_1, \ldots, x_n)$ is defined belongs to $U$. It is now routine to verify that $h_\bullet$ provides a contracting homotopy, because for any \emph{given} $x_1, \ldots, x_n$ and $x$ as above, the points $x, x_1, \ldots, x_n$ determine a full simplex in $X^\perp_\bullet$. Explicitly, recall that the differential $\lin(X^\perp_{n-1}) \to \lin(X^\perp_n)$ is the alternating sum of the maps $d_{n,i}$ over $0\leq i \leq n$, where $d_{n,i}$ omits the $i$th element of a $(n+1)$-tuple. Then the relations $d_{n,i} h_n = h_{n+1} d_{n+1, i+1}$ and $h_{n+1} d_{n+1, 0} = \mathrm{Id}$ hold, which shows that $h_\bullet$ is a contracting homotopy.
\end{proof}

Let $G$ be a group acting on a semi-simplicial set $X_\bullet$. This gives rise to hypercohomology spectral sequences relating the bounded cohomology of $G$ to that of $X_\bullet$ and of its quotient $X_\bullet/G$, stated in \Cref{thm:sss} below. Some care is needed since a number of standard cohomological techniques do not hold for bounded cohomology. \Cref{defn:modulus} affords us a quantitative control.

\begin{thm}\label{thm:sss}
Let $G$ be a group acting on a boundedly acyclic connected semi-simplicial set $X_\bullet$ and let $N\in\bN\cup\{\infty\}$. We make the following assumptions for each $0\leq p< N$:
\begin{enumerate}[(i)]
\item The stabiliser of any point in $X_p$ has vanishing $H^q_b$ for all $q>0$ with $p+q<N$.\label{pr:ss:acyclic}
\item Given $q>0$ with $p+q<N$, the $q$th vanishing moduli of all those stabilisers are uniformly bounded.\label{pr:ss:ub}
\end{enumerate}
Then there is an isomorphism $H^p_b(G) \cong H^p_b(X_\bullet /G) $ for every $0\leq p < N$.
\end{thm}

\begin{rem}\label{rem:ss}
In order to demystify the condition~\eqref{pr:ss:ub}, we note that it is obviously satisfied when for instance there are finitely many group isomorphism classes among the stabilisers in $G$ of points in $X_p$. In the even more special case where there are only finitely many conjugacy classes in $G$ of such stabilisers, the above theorem admits a simpler proof which does not explicitly involve vanishing moduli. In any case, we shall also need to apply it to situations with infinitely many isomorphism classes of stabilisers.
\end{rem}

The technical assumption~\eqref{pr:ss:ub} is made to ensure the uniform boundedness entering the following.

\begin{lem}\label{lem:uniform}
Under the assumptions of \Cref{thm:sss}, consider the $G$-representation on $\lin(X_p)$. Then $H^q_b(G,\lin(X_p))$ vanishes for all $q>0$ with $p+q<N$.
\end{lem}

\begin{proof}[Proof of \Cref{lem:uniform}]
Fix $p$ and $q$ as in the statement. Let $J\se X_p$ be a set of
representatives of the $G$-orbits and let $G_j<G$ be the stabiliser of
$j\in J$. Thus $X_p$ can be identified with $\bigsqcup_{j\in J} G/G_j$.

Consider now a cocycle $c$ representing an element $H^q_b(G,\lin(X_p))$;
that is, $c$ is a $G$-equivariant bounded cocycle $c\colon G^{q+1} \to
\lin(X_p)$. The above decomposition of $X_p$ decomposes $c$ as into
$G$-equivariant bounded cocycles $c_j\colon G^{q+1} \to \lin(G/G_j)$ and
moreover $\| c\|_\infty = \sup_{j\in J} \| c_j\|_\infty$. In order to
prove the statement, it suffices to show that each $c_j$ is a coboundary
of some bounded $G$-equivariant $b_j\colon G^{q} \to \lin(G/G_j)$ in
such a way that $ \| b_j\|_\infty$ is bounded \emph{independently} of
$j$; indeed in that case the map $b\colon G^{q} \to \lin(X_p)$ defined
by these $b_j$ will itself be bounded and witness $c=d b$. We thus fix
some $j\in J$ and proceed to show that fact:

\textbf{Claim.} Given $c_j$ as above, consider the corresponding
$G_j$-invariant map $\overline{c_j} \colon G^{q+1} \to \mathbf{R}$
defined by $\overline{c_j} (g_0, \ldots, g_q) = c_j(g_0, \ldots, g_q)
(e)$. Then $c_j \mapsto \overline{c_j}$ is an isometric isomorphism on
the cochain level from the standard resolution for
$H^q_b(G,\lin(G/G_j))$ to the resolution for $H^q_b(G_j)$ given by
bounded $G_j$-invariant cochains on $G^{q+1}$ (as in
\Cref{lem:mod:subgroup}).

This claim is a very explicit form of the Eckmann--Shapiro isomorphism
for bounded cohomology, on the cochain level, and it is  proved e.g.
in~\cite[10.1]{MonodLNM}. With this claim in hand, the assumptions of
\Cref{thm:sss} can now be applied to obtain $b_j\colon G^{q} \to
\lin(G/G_j)$ with $ \| b_j\|_\infty$ bounded independently of $j$, as
desired.
\end{proof}

\begin{proof}[Proof of \Cref{thm:sss}]
We consider the double complex
 \[
 L^{p,q}\coloneqq \lin(G^{p+1}\times X_q)^G \cong \lin\left(G^{p+1}, \lin(X_q)\right)^G.
 \]
The spectral sequence associated to the horizontal filtration of $L^{\bullet, \bullet}$ has $E_1$-page defined by the cohomology of
\[
\lin\left(G^{p+1}, \lin (X_{q-1})\right)^G \lra \lin\left(G^{p+1}, \lin (X_{q})\right)^G \lra \lin\left(G^{p+1}, \lin (X_{q+1})\right)^G 
\]
Since $X_\bullet$ is boundedly acyclic, this cohomology vanishes for all $p\geq 0$ and all $q>0$ because for fixed $p$ the functor $\lin(G^{p+1}, -)^G$ is exact for dual morphisms (see e.g.~\cite[8.2.5]{MonodLNM}, wherein the second countability assumption is irrelevant in our setting since we consider $G$ without topology). For $q=0$, we obtain $\lin(G^{p+1})^G$ by connectedness of $X_\bullet$. It follows that the spectral sequence converges to $H^\bullet_b(G)$.

Turning to the spectral sequence associated to the vertical filtration, the $E_1$-page is
\[
E_1^{p,q}=H^q_b\left(G, \lin (X_{p}) \right).
\]
Applying \Cref{lem:uniform} to $X=X_p$, the assumption~\eqref{pr:ss:acyclic} and~\eqref{pr:ss:ub} imply $E_1^{p,q}=0$ for all $q>0$ with $p+q<N$.

As for $E_1^{p,0}$, we have
\[
E_1^{p,0}= \lin (X_{p})^G \cong \lin (X_{p} /G)
\]
which implies
\[
E_2^{p,0} \cong H^p_b (X_{\bullet} /G).
\]
Putting everything together, we have as claimed $H^p_b(G) \cong  H^p_b (X_{\bullet} /G)$ for  $0\leq p < N$.
\end{proof}

\section{Homeomorphisms and diffeomorphisms of the circle}\label{sec:circle}
The case of the circle introduces already some of the ideas that will be used repeatedly in this article, but it evades technical considerations such as moduli of vanishing or difficult stabilisers. In fact, even the generic relation below is not strictly needed in this case, although it does simplify the set of stabilisers and the quotient.

We begin with a general definition: a \textbf{fat point} in an oriented $n$-manifold $M$ refers to a germ at~$0$ of an orientation-preserving embedding $B^n \to M$, where $B^n\se \bR^n$ is the open unit disc. The image of~$0$ in $M$ is the \textbf{core} of the fat point.

We shall prove \Cref{circle} using a semi-simplicial set of fat points. Let thus $X$ be the set of fat points in $S^1$. Given $x,y\in X$, we write $x\perp y$ if $x$ and $y$ have distinct cores. This is a generic relation; therefore, Proposition~\ref{prop:generic} implies that $X_\bullet^\perp$ is a boundedly acyclic connected semi-simplicial set.

\begin{proof}[Proof of \Cref{circle}]
Since the action of $G=\Homeo_\circ(S^1)$ on $X$ preserves the relation $\perp$, it induces an action on the semi-simplicial set $X_\bullet^\perp$. We claim that the conditions of Theorem~\ref{thm:sss} are satisfied for $N=\infty$.

Regarding condition~\eqref{pr:ss:acyclic}, observe that the stabiliser in $G$ of a fat point is isomorphic to the group $\Homeo_c(\bR)$ of compactly supported homeomorphisms, which is boundedly acyclic by Matsumoto--Morita's \Cref{thm:MM}.

More generally, let $L$ be the stabiliser of an element of $X_p^\perp$. The cores to the fat points determine $p+1$ components of the circle and these components cannot be permuted by $L$ since every element of $L$ is trivial in some neighbourhood of the cut-points. Thus $L$ is isomorphic to the direct product of $p+1$ copies of $\Homeo_c(\bR)$. Therefore, in view of Proposition~\ref{prop:ext:van}, this subgroup is also boundedly acyclic. 

As to condition~\eqref{pr:ss:ub}, the above discussion of stabilisers already establishes that there are finitely many isomorphism types for each given $p$. We observe, with \Cref{rem:ss} in mind, that we have the stronger fact that there are only finitely many $G$-orbits in every given $X_p^\perp$. Indeed, the orbit of a tuple $(x_0, \ldots x_p)$ is determined by the cyclic ordering of the cores $\dot x_i$ of the $x_i$. To see this, consider another tuple $(y_0, \ldots y_p)$ such that the $\dot y_i$ have the same cyclic order as the $\dot x_i$. Choose orientation-preserving embeddings $\xi_i\colon B^1\to S^1$ representing $x_i$ in such a way that all images $\xi_i(B^1)$ are pairwise disjoint, which is possible since the $\dot x_i$ are distinct. Define similarly $\eta_i$ for $y_i$. We have partial homeomorphisms $\eta_i\circ \xi_i^{-1}$, which can be completed to a homeomorphism $h\in G$ since we have disjoint images in the same cyclic order. By construction, $h(x_i)=y_i$ holds for all $i$.

In conclusion, Theorem~\ref{thm:sss} implies that $H^\bullet_b(G)$ is isomorphic to the bounded cohomology of the semi-simplicial set $X_\bullet^\perp/G$, which coincides with the usual cohomology $H^\bullet(X_\bullet^\perp/G)$ since all orbits sets $X_p^\perp/G$ are finite. In fact we showed more, namely that $X_\bullet^\perp/G$ is the complex of cyclic orderings of finite sets, which is well-known and easily verified to have $\bR[\Euler]$ as its cohomology ring, see e.g.~\cite[\S5]{MR2871163}. Alternatively, this can be seen by reversing the above argument, but in usual cohomology (see \cite[\S3.1.3]{nariman2020local}). That is to say, $H^\bullet(X_\bullet^\perp/G)$ is isomorphic to $H^\bullet(G)$ because one checks that $X_{\bullet}^\perp$ is a contractible semisimplicial set, and because the stabilisers are acyclic (in the usual sense). The latter fact follows very much like in the bounded case above: it is reduced to the acyclicity of $\Homeo_c(\bR)$ established by Mather~\cite{MR0288777} by using the K\"unneth formula.

Finally, the fact that $H^\bullet(G)$ is isomorphic to $\bR[\Euler]$ follows from Thurston's theorem (\cite[Cor.~(b) of Thm.~5]{thurston1974foliations}) as recalled in the introduction. All the above arguments being completely natural and given by a morphism of spectral sequences (the one for bounded cohomology is mapped to the ordinary one by the forgetful functor), the isomorphism preserves also the ring structure of [bounded] cohomology.
\end{proof}

\begin{proof}[Proof of \Cref{circle:diff}]
We now turn to the $C^r$-diffeomorphism group $G^r=\Diff^r_\circ(S^1)$, where $r\in\bN\cup\{\infty\}$. In analogy with the topological case, we define $C^r$-fat points as germs of  $C^r$-embeddings. This yields a corresponding semi-simplicial $G^r$-set $C^r X_\bullet^\perp$ which is still boundedly acyclic by Proposition~\ref{prop:generic}. The quotient $C^r X_\bullet^\perp/G^r$ is isomorphic to the quotient $X_\bullet^\perp/G$ of the proof of \Cref{circle} because the same transitivity property holds, namely: the orbit of a tuple is determined by the cyclic order of the cores of the fat points constituting the tuple. A priori, the only substantial difference between the two settings resides with the stabilisers, which are products of copies of $\Diff^r_c(\bR)$. Indeed, recall that $\Diff^r_c(\bR)$ is not an acyclic group for $r>1$ since there is the Godbillon--Vey class $\GV\in H^2(\Diff^r_c(\bR))$ which is nontrivial. However, this group is still boundedly acyclic by \Cref{thm:diff:ba} below. Therefore we conclude as above that $H_b^\bullet(G)$ is isomorphic to $\bR[\Euler_b]$.

In order to complete the proof of \Cref{circle:diff} in the $C^r$ case, is only remains to justify that the isomorphism is induced by the restriction map associated to the inclusion $\iota$ of $\Diff^r_\circ(S^1)$ into $\Homeo_\circ$. This is the case because the comparison between the two proofs goes through the $\iota$-equivariant semi-simplicial inclusion morphism $C^r X_\bullet^\perp$, which induces the identification between $C^r X_\bullet^\perp/G^r$ and $X_\bullet^\perp/G$.

Finally, in the PL case, we consider likewise PL fat points and only need to justify that $\PL_c(\bR)$ is boundedly acyclic, which we do in \Cref{lem:PL:ba} below. This completes the proof of \Cref{circle:diff}.

(In fact the PL case is even simpler because it is equally easy to show the bounded acyclicity of $\PL(\bR)$, so that fat points can be replaced by usual points in that instance.)
\end{proof}

This proof can further be adapted to other circle transformation groups as long as the main ingredients are preserved. Specifically, suppose that a group $G<\Homeo_\circ(S^1)$ acts transitively on all cyclically oriented tuples of points in some $G$-orbit, or on some fat version thereof. If the stabiliser of any such tuple (respectively fat tuple) is boundedly acyclic, then $H^\bullet_b(G)$ is isomorphic to $\bR[\Euler_b]$. We record the special case of Thompson's group $T$ (see~\cite{Cannon-Floyd-Parry} for a definition), where the bounded acyclicity of stabilisers is proved in~\cite{Monod_wreath}.

\begin{cor}
The bounded cohomology ring $H^\bullet_b(T)$ of Thompson's group $T$ is isomorphic to $\bR[\Euler_b]$.\qed
\end{cor}

Other examples of groups for which this argument holds (again using~\cite{Monod_wreath}) are piecewise-projective circle groups discussed in~\cite{Monod_PNAS}.

\medskip

We now justify the bounded acyclicity that was used in the proof of \Cref{circle:diff} for various compactly supported transformation groups of $\bR$. We establish this in the following more general setting, as announced in \Cref{thm:MR^n}.

\begin{thm}\label{thm:diff:ba}
Let $n\in \bN$ and consider $Z=\bR^n$, or more generally $Z=M\times \bR^n$ for any closed manifold $M$ when $n>0$.

Then the groups $G=\Homeo_c(Z)$, $\Homeo_{c, \circ}(Z)$, $\Diff^r_c(Z)$ and $\Diff^r_{c, \circ}(Z)$ are boundedly acyclic for all $r\in\bN\cup\{\infty\}$.
\end{thm}

\begin{proof}
In order to apply \Cref{thm:wreath}, we need to check the two conditions of that theorem. Let $Z_0 = M \times D^n$, where $D^n$ is the closed unit ball in $\bR^n$.

To verify the first condition, it suffices to show that every compact set $C\se M\times \bR^n$ can be mapped into $Z_0$ by some element $h\in G_\circ$. We can take $h=\mathrm{Id}_M\times h'$, where $h'$ is a homothety on some large ball to ensure $h'(C) \se D^n$, and then radially smoothen it to the identity away from a larger ball.

The second condition postulates the existence of $g$ in $G$ or $G_\circ$ such that $g^p(Z_0)$ is disjoint from $Z_0$ for every integer $p\geq 1$. For $n=1$ and $M$ trivial, take first $g_0$ to be any ``bump shift'', that is, a transformation which is strictly increasing on some bounded open interval $I$ containing the interval $D^1$. Then any sufficiently high power $g$ of $g_0$ will have the required property since $g_0$ is order-preserving. For $n>1$, we can use the one-dimensional case in one coordinate and suitably smoothen it out to the identity along the other coordinates. In the case of $M\times \bR^n$, we extend $g$ by the identity on the $M$ coordinate.
\end{proof}

Since this proof reduces the statement to \Cref{thm:wreath}, we can combine it with \Cref{lem:inf:power}; we record this as follows:

\begin{cor}\label{cor:MM:power}
The powers $\Homeo_c(\bR^n)^\bN$, $\Diff^r_c(\bR^n)^\bN$ and $\Diff^r_{c,\circ}(\bR^n)^\bN$ are boundedly acyclic. The same holds more generally for $M\times \bR^n$ instead of $\bR^n$, where $M$ is any closed manifold and $n>0$.\qed
\end{cor}

If in the proof of \Cref{thm:diff:ba} we replace the smooth bump shift by a PL bump shift and the (local) homothety by an analogous PL shrinking map, we obtain:

\begin{lem}\label{lem:PL:ba}
The group $\PL_c(\bR^n)$ is boundedly acyclic for all $n\in \bN$.\qed
\end{lem}

\section{Homeomorphisms and diffeomorphisms of the disc}\label{disk}
In this section, we further leverage the method of generic semi-simplicial sets to prove a bounded version in dimension $n=2$ of McDuff's theorem (\cite[Cor.~2.13]{mcduff1980homology}) stating that the restriction homomorphism
 \[
 \Homeo_\circ(D^n)\to \Homeo_\circ(S^{n-1})
 \]
 induces a cohomology isomorphism:
 
 \begin{thm}\label{mcduff}
 The restriction map $ \Homeo_\circ(D^2)\to \Homeo_\circ(S^{1})$ induces an isomorphism in bounded cohomology.
 \end{thm}

Thus we can apply \Cref{circle} and deduce:

\begin{cor}
$H^{\bullet}_b(\Homeo_\circ(D^2))\cong \bR[\Euler_b]$.\qed
\end{cor}

\begin{rem}\label{Calegari}
 As a consequence of McDuff's result(\cite[Cor.~2.12]{mcduff1980homology}) at least in dimensions $n\leq 3$ we know that the inclusion $\Homeo_\circ(D^n)\hookrightarrow \Homeo_\circ(B^n)$ induces an isomorphism on group cohomology. Calegari (\cite[\S4.1]{calegari2004circular}) however proved that the Euler class in $H^2(\Homeo_\circ(B^2))$ is {\it not} a bounded class. Therefore, the inclusion  $\Homeo_\circ(D^2)\hookrightarrow \Homeo_\circ(B^2)$ does not induce an isomorphism on {\it bounded} cohomology.
It would be interesting to see if $\Homeo_\circ(D^3)$ has a non-trivial bounded cohomology.
\end{rem}

To find a suitable resolution, we shall define semi-simplicial sets of fat chords:

\begin{defn}
A \textbf{chord} is an embedding of the closed arc $D^1$ into $D^2$ such that the endpoints $\partial D^1$, but no other points of $D^1$, are mapped to the circle $\partial D^2$.

We fatten this definition by considering an orientation-preserving embedding $\varphi: D^{1}\times \bR\hookrightarrow D^{2}$ such that $\varphi$ restricts to an embedding of $(\partial D^{1})\times \bR$ into $\partial D^2$ and such that $\varphi(B^{1}\times \bR)$ lies in $B^2$. Then let $[\varphi]$ denote the germ of $\varphi$ around $D^{1}\times \{0\}$, i.e.\ we say that $\varphi$ and $\psi$ have the same germ if the restrictions of $\varphi$ and $\psi$ to $D^{1}\times (-\epsilon,\epsilon)$ are equal for some unspecified  positive $\epsilon$. We call the germ $[\varphi]$ a \textbf{fat chord} and denote the set of fat chords by $\FCh$. The chord obtained by restricting $\varphi$ to $D^{1}\times \{0\}$ depends on $[\varphi]$ only and is called the \textbf{core} of $[\varphi]$. See \Cref{fig:fatarcs}.
 \end{defn}

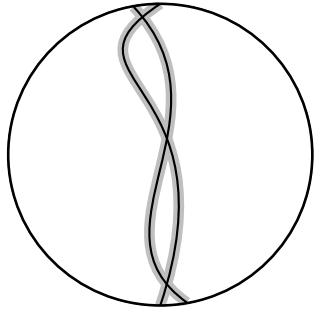
\begin{figure}[h]
 \begin{tikzpicture} 
\draw  [line width=4, gray!50] (0.35,-1.96) .. controls (-1,-1) and (1, 0.5)..  (-0.35,1.96);
\draw  [line width=4, gray!50] (0,-2) .. controls (1,1) and (-1.5, 1)..  (0,2);
\draw  [line width=.8] (0,-2) (0.35,-1.96) .. controls (-1,-1) and (1, 0.5)..  (-0.35,1.96);
\draw  [line width=0.8] (0,-2) (0,-2) .. controls (1,1) and (-1.5, 1)..  (0,2);
  \draw [line width=1] (0,0) circle (2cm);
\draw [line width=1.5, white] (0,0) circle (2.05cm); 
\end{tikzpicture}  
    \caption{Two fat chords in $D^2$ with their cores.}\label{fig:fatarcs}
    \label{core}
 \end{figure}

 \begin{defn}\label{defn:transverse:chords}
We call two fat chords \textbf{strictly transverse} if
\begin{itemize}

\item their cores do intersect, and
\item their cores are topologically transverse, and
\item the endpoints of their cores are all distinct.
\end{itemize}  
 \end{defn}

Consider a finite set of fat chords that are pairwise strictly transverse. Then the complement in $B^2$ of their cores is a finite union of open discs. The boundary of each such disc is partitioned into pieces of chords and pieces of the the original boundary $S^1= \partial D^2$. The point of the first condition in \Cref{defn:transverse:chords} is that there is at most one component from $S^1$ for each of these open discs in this partition. This restriction will later allow us to show that the stabilisers described in the next lemma are boundedly acyclic.

We write $\Homeo(D^2; \rel D)$ for the subgroup of $\Homeo(D^2)$ consisting of homeomorphisms that are trivial in some neighbourhood in $D^2$ of a boundary interval $D\se \partial D^2$, that is, and embedded interval $D^1\cong D \se \partial D^2$. More precisely, this definition understands that an arbitrary such $D$ has been chosen, but that each element of $\Homeo(D^2; \rel D)$ can be trivial on a different neighbourhood of $D$ in $D^2$. The above definitions thus imply the following description.

\begin{lem}\label{stab:fat:chords}
The stabiliser in $\Homeo_\circ(D^2)$ of any $(n+1)$-tuple of pairwise strictly transverse fat chords is a finite product of groups, each isomorphic either to $\Homeo(D^2; \rel S^1)$ or to $\Homeo(D^2; \rel D)$. Moreover, there are at most $2n+2$ factors of the latter type.\qed
\end{lem}

\noindent
On the other hand, given that chords may intersect a large number of times, there is no bound on the number of the factors of the $\Homeo(D^2; \rel S^1)$ type.

\begin{rem}\label{stab:fat:chords:rel}
The description recorded in \Cref{stab:fat:chords} simplifies for the action of the subgroup $\Homeo(D^2; \rel S^1)$. In that case, the stabiliser of any tuple of pairwise strictly transverse fat chords is a finite power of $\Homeo(D^2; \rel S^1)$.
\end{rem}

Next, we show that as the name suggests, strict transversality satisfies the condition of \Cref{defn:generic}.

\begin{lem}\label{chords:generic}
The relation of strict transversality is generic on $\FCh$.
\end{lem}

\begin{proof}
Given finitely many fat chords $a_1, \ldots, a_k$, the claim is that there is $a_0\in\FCh$ strictly transverse to $a_i$ for all $i>0$. We first outline the strategy of the proof, writing $\dot a_i$ for the core of $a_i$.

There is no difficulty whatsoever in satisfying the first and last conditions of \Cref{defn:generic}. The substance of the lemma lies in ensuring the second condition, namely that the core $\dot a_0$ of the desired $a_0$ intersects transversely \emph{all} given $\dot a_i$. Adjusting $a_0$ to be transverse to one $a_i$ at a time can immediately be obtained by Schoenflies's theorem applied to $\dot a_i$, but the difficulty is that the various $\dot a_i$ might intersect each other non-transversally, for instance in Cantor sets: we remind the reader that for $i,j>0$, no transversality assumption is made for $a_i$ relatively to $a_j$.

The strategy is to construct $a_0$ inductively, starting from a suitable fat chord $a_{0}^0$ satisfying the first and last conditions of \Cref{defn:generic}. We then use that the second condition is in fact ``generic'' in a topological sense. This allows us to perform a small perturbation (in the $C^0$ sense) of our first choice $a_{0}^0$, to choose inductively $a_{0}^i$ so that it is transverse to all $\{a_j\}_{j=1}^i$. We then define $a_0=a_0^k$. This topological genericity of transversality for arcs is a standard fact, nicely explained in~\cite{mathoverflow2}; we shall recall the details below.

\smallskip

The base $a_{0}^0$ of the induction must be chosen so that small $C^0$ perturbations do not make the intersections of cores empty. To that end, we choose $a_0^0$ in such a way that for each $i>0$ \emph{some sub-arc} of the core $\dot a_0^0$ intersects $\dot a_i$ transversally. This can be done in steps as follows. Start from a boundary point distinct from all endpoints. By Schoenflies's theorem, $\dot a_1$ divides the disc into two discs and thus we can choose an initial (fat) arc from the boundary point to a point in the other disc component, crossing $\dot a_1$ once transversally and stopping at an interior point not lying on any $\dot a_i$. We then apply Schoenflies's theorem to the next core $\dot a_2$ and cross it transversally in the same fashion, repeating this argument until $\dot a_k$ is crossed and ending on a new boundary point distinct from all endpoints. We need to avoid self-intersections of $\dot a_0^0$, but this is possible since at every intermediate step the partially constructed arc $\dot a_0^0$ is retractable to the boundary.

\smallskip

Now we prepare the inductive argument. Given $i,j>0$, consider the pairwise intersection $\dot a_i \cap \dot a_j$ of the cores and denote by $K_{ij}$ the topological boundary $K_{ij}= \partial_{\dot a_i}(\dot a_i \cap \dot a_j)$ in $\dot a_i$ of this intersection. Notice that $K_{ij}= K_{ji}$ holds because interior points of the intersection with respect to either $\dot a_i$ or $\dot a_j$ coincide. Consider the union $K_i=\bigcup_{j} K_{ij}$. Since the boundary of a closed set is always nowhere dense, $K_i$ is a nowhere dense subset of the arc $\dot a_i$.

The inductive claim at step $i$ (with $1\leq i \leq k$) is that $a_0^{i-1}$ admits an arbitrarily small perturbation $a_0^i$ such that for all $1\leq j \leq i$, the core $\dot a_0^i$ is transverse to $\dot a_j$ and avoids the set $K_j$.

The proof of this claim for $i=1$ is a simpler version of the proof for $i>1$, so we establish the latter while assuming that it already holds for all $1\leq j < i$. Consider
\[
N= (\dot a_0^{i-1} \cap \dot a_i ) \setminus \bigcup_{1\leq j < i} (\dot a_{j} \cap \dot a_i ).
\]
Since $\dot a_0^{i-1}$ avoids $K_{ji}=K_{ij}$, we have in fact
\[
N= (\dot a_0^{i-1} \cap \dot a_i ) \setminus \bigcup_{1\leq j < i} \mathrm{Int}_{a_i}(\dot a_{j} \cap \dot a_i )
\]
so that $N$ is a closed set and hence it has a neighbourhood $V$ separating it from $a_j$ for all $1\leq j < i$. Now each sub-arc of $\dot a_0^{i-1}$ in $V$ can be perturbed within $V$ to become transverse to $\dot a_i$ using Schoenflies's theorem, and we can as well do it with the corresponding portion of the fat arc $a_0^{i-1}$. We can additionally ensure that the core of this new perturbation $a_0^{i}$ avoids $K_i$ since $K_i$ is nowhere dense in the sub-arc of $\dot a_i$ that we are crossing.

The reason that this completes the induction step is that outside $N$, the previous iteration $\dot a_0^{i-1}$ was already transverse to $\dot a_i$ because any intersection with $\dot a_i$ in $N$ would already occur in the interior of $\dot a_i \cap \dot a_j$ for some $j<i$. Now that the induction is complete, the final iteration  $a_0^{k}$ is the desired $a_0$.
\end{proof}

We consider now the semi-simplicial set $\FCh_\bullet^\perp$ of tuples of pairwise strictly transverse fat chords. Then \Cref{chords:generic} allows us to apply Proposition~\ref{prop:generic} and deduce:

\begin{cor}\label{FCh:ba}
The semi-simplicial set $\FCh_\bullet^\perp$ is boundedly acyclic.\qed
\end{cor}

We shall also need to consider a somewhat technical subset of $\FCh$ and hence of $\FCh_\bullet^\perp$, as follows.

\begin{defn}
A fat chord $[\varphi]$ is called \textbf{radial} (near the boundary) if we can choose $\varphi: D^{1}\times \bR\hookrightarrow D^{2}$ which is radial in a neighbourhood of $\{\pm 1\}\times \bR$ in the following sense: there exists $\epsilon >0$ such that
\[
\varphi(\pm r, t) = r \varphi(\pm 1, t) \kern5mm \forall\, 1-\epsilon < r \leq 1, \forall\, t.
\]
We write $\RFCh \se \FCh$ for the set of radial fat chords and $\RFCh^\perp_\bullet$ for the corresponding sub-semi-simplicial set of $\FCh^\perp_\bullet$.
\end{defn}

This specific definition is somewhat arbitrary since we are working in the topological category; the point is only to choose some normalisation of the germ near the boundary. We note that that the arguments given for $\FCh$ can be repeated virtually unchanged to yield the following variation of \Cref{FCh:ba}.

\begin{cor}\label{RFCh:ba}
The semi-simplicial set $\RFCh_\bullet^\perp$ is boundedly acyclic.\qed
\end{cor}

\begin{rem}
It is plausible to use more high powered transversality results in the topological settings (\cite{MR0645390}) to prove an analogue of \Cref{FCh:ba} and \Cref{RFCh:ba} in higher dimensions. We will, however, only be able to work in dimension two because the fact that the complementary regions in $D^2$ of simplices of $\FCh^\perp_\bullet$ is again a union of discs is, as we shall see,  an important feature in the proof \Cref{mcduff}.
 \end{rem}

The relevance of the radial semi-simplicial subset hinges on the following observation:

\begin{lem}\label{lem:isom:quot}
The inclusion map $\RFCh \to \FCh$ descends to a bijection
\[
\RFCh / \Homeo(D^2; \rel S^1) \xrightarrow{\ \cong\ }  \FCh / \Homeo(D^2; S^1).
\]
More generally, for all $p$ it induces a bijection
\[
\RFCh^\perp_p / \Homeo(D^2; \rel S^1) \xrightarrow{\ \cong\ } \FCh^\perp_p / \Homeo(D^2; S^1).
\]
\end{lem}

\begin{proof}
By the isotopy extension theorem for homeomorphisms in dimension~$2$ (see \cite[Thm.~1.7.10]{hamstrom1974homotopy}) and the fact that $\Homeo^{\topo}(D^2; S^1)$ as a topological group is contractible (\cite[Thm.~1.5.2]{hamstrom1974homotopy}), we know that space of embeddings of chords with fixed ends in $D^2$ is contractible. In particular, if two fat chords have the same ends, they are isotopic relative to their ends. Hence, by an isotopy fixing the boundary pointwise, we can make the fat chords radial near the boundary and in fact we can make the support of this isotopy in any given neighbourhood of the boundary. Hence the above maps are surjective. For injectivity, note that if two radial fat chords have the same ends, they overlap near the boundary so they will be isotopic by an isotopy that is identity near the boundary. 
\end{proof}

In order to prove \Cref{mcduff}, we still need to deal with an auxiliary subgroup. We recall the notation $\Homeo(D^2; \rel D)$ introduced before \Cref{stab:fat:chords} for some boundary interval $D\se S^1$.

\begin{lem}\label{lem:halfcircle}
The group $\Homeo(D^2; \rel D)$ satisfies the assumptions of \Cref{thm:wreath} and hence is boundedly acyclic.  
\end{lem}
 \begin{figure}[h]\label{nestedsquares}
\begin{tikzpicture}[scale=1.9]
  \def\rectanglepath{-- ++(1cm,0cm)  -- ++(0cm,1cm)  -- ++(-1cm,0cm) -- cycle}
    \def\rectanglepathsmall{-- ++(0.5cm,0cm)  -- ++(0cm,0.5cm)  -- ++(-0.5cm,0cm) -- cycle}
        \def\rectanglepathsmalll{-- ++(0.2cm,0cm)  -- ++(0cm,0.2cm)  -- ++(-0.2cm,0cm) -- cycle}

  \draw (0,0) \rectanglepath;
  \draw (0.1,0) \rectanglepathsmall;
    \draw (0.7,0) \rectanglepathsmalll;
  \node at (0.35,0.2) {$Z_0$};
    \node at (0.5,0.7) {$Z$};
 \end{tikzpicture}
\caption{$Z$, $Z_0$ and the first displaced copy $g Z_0$.} 
\end{figure}
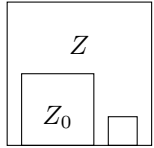
\begin{proof}[Proof of \Cref{lem:halfcircle}]
It suffices to justify that we are in the situation of \Cref{thm:wreath}. To do so, we consider the following equivalent set-up (see Figure~\ref{nestedsquares}). Let $Z$ be a square and let $D$ be the union of three edges of this square. We call the remaining edge the \textbf{free edge} since elements of $\Homeo(D^2; \rel D)$ can restrict to non-trivial homeomorphisms of that edge. An important point is that each such edge homeomorphism will be supported in a compact subset of the interior of that edge.

Let now $Z_0$ be a smaller square inside $Z$ whose free edge lies in the interior the free edge of $Z$. Then indeed we can find a ``displacement'' $g$ and both conditions of \Cref{thm:wreath} are satisfied.
\end{proof}

\begin{rem}
In the first version of this article, we observed that \Cref{lem:halfcircle} can be proved using the original Matsumoto--Morita method.
\end{rem}

At this point, we can conclude from the description given in \Cref{stab:fat:chords} that the stabiliser of any $(n+1)$-tuple of pairwise strictly transverse fat chords is a boundedly acyclic group since bounded acyclicity passes to products (\Cref{prop:ext:van}). However, we shall need to know that the vanishing moduli of this bounded acyclicity does not depend on the combinatorics of chord intersections even though this combinatorics leads to unboundedly many factors in the product described in \Cref{stab:fat:chords} even if we fix $n$. The desired uniformity is a consequence of \Cref {prop:power:mod}, as follows.

\begin{cor}\label{cor:stab:fat:chords}
For every $q>0$ there is a constant bounding the $q$th vanishing modulus of the stabiliser in $\Homeo_\circ(D^2)$ or in $\Homeo(D^2; \rel S^1)$ of any $(n+1)$-tuple of pairwise strictly transverse fat chords.
\end{cor}

\begin{proof}
Let $G$ be the product of countably many copies of $\Homeo(D^2; \rel S^1) \cong \Homeo_c(\bR^2)$ and of $\Homeo(D^2; \rel D)$. Both types of factors satisfy the assumptions of \Cref{thm:wreath} and hence \Cref{lem:inf:power} implies that $G$ is boundedly acyclic. According to \Cref{stab:fat:chords} and \Cref{stab:fat:chords:rel}, all the stabilisers considered in the statement of \Cref{cor:stab:fat:chords} are subproducts of $G$. Therefore, the statement follows indeed from \Cref{prop:power:mod}.
\end{proof}

\begin{proof}[Proof of \Cref{mcduff}]
For shorter notation, we write
\[
G= \Homeo_\circ(D^2) \kern3mm\text{and}\kern3mm G'  = \Homeo(D^2;S^1)
\]
so that the restriction to $S^1$ yields an identification $G/G' = \Homeo_\circ(S^1)$.

We claim that the semi-simplicial set $\FCh^\perp_\bullet / G'$ is boundedly acyclic. In view of \Cref{lem:isom:quot}, it suffices to establish this for $\RFCh^\perp_\bullet / \Homeo(D^2; \rel S^1)$ instead. We shall deduce this from \Cref{thm:sss}. To that end, recall that $\Homeo(D^2; \rel S^1) \cong \Homeo_c(\bR^2)$ is boundedly acyclic by Matsumoto--Morita's theorem, while $\RFCh^\perp_\bullet$ is boundedly acyclic by \Cref{RFCh:ba}. It thus remains to check the conditions on the stabilisers, which are granted by the \Cref{cor:stab:fat:chords}. This confirms the claim.

The claim puts us in the position to apply \Cref{thm:sss} a second time, but to the action of $G/G'$ on $\FCh^\perp_\bullet / G'$. The stabiliser of a $(p+1)$-tuple for this action corresponds to the stabiliser of $2(p+1)$ generic fat points in $S^1$ under the identification of $G/G'$ with $\Homeo_\circ(S^1)$. Thus, as in the Proof of \Cref{circle}, these stabilisers are boundedly acyclic (and there is only one isomorphism type when $p$ is fixed). Therefore, \Cref{thm:sss} provides an isomorphism between the bounded cohomology of $\Homeo_\circ(S^1)$ and the bounded cohomology of the quotient of $\FCh^\perp_\bullet / G'$ by $G/G'$, which is none other than $\FCh^\perp_\bullet / G$.

We invoke a third time \Cref{thm:sss}, now for the $G$-action on $\FCh^\perp_\bullet$. That semi-simplicial set is boundedly acyclic by \Cref{FCh:ba}. To justify the assumptions on the stabilisers, we recall the description of \Cref{stab:fat:chords} and invoke again \Cref{cor:stab:fat:chords}. We thus obtain an isomorphism between the bounded cohomology of $G$ and of $\FCh^\perp_\bullet / G$, which we previously identified with the bounded cohomology of $\Homeo_\circ(S^1)$.

It only remains to justify that this isomorphism is indeed induced by the restriction map $G \to \Homeo_\circ(S^1)$, or equivalently by the quotient map $G\to G/G'$. The two isomorphisms produced by the second and third applications of \Cref{thm:sss} arise from parallel spectral sequences, connected by the morphism of spectral sequences induced by the quotient maps $\FCh^\perp_\bullet \to \FCh^\perp_\bullet / G'$ and $G\to G/G'$. Thus indeed the isomorphism is induced by restriction (which in particular preserves the ring structure of bounded cohomology); this completes the proof of \Cref{mcduff}.
\end{proof}

\begin{proof}[Proof of \Cref{thm:disk:diff}]
Our proof of \Cref{mcduff} is set up in such a way that it can be followed equally well in the $C^r$ case. We only need to replace our semi-simplicial sets $\FCh^\perp_\bullet$ and $\RFCh^\perp_\bullet$ by the corresponding sets of $C^r$ germs. The main difference, viewed from the perspective of ordinary cohomology, is that the bounded acyclicity of stabilisers still holds but cannot be traced back to Matsumoto--Morita methods. Instead, we argue that \Cref{lem:halfcircle} holds unchanged for (the connected component of) diffeomorphism groups because it relies on \Cref{thm:wreath} and the latter does not discriminate according to regularity. The same holds for the bounded acyclicity of $\Diff^r_\circ(D^2; \rel S^1)$.
\end{proof}

\section{Homeomorphisms and diffeomorphisms of the sphere \texorpdfstring{$S^n$}{Sn}}\label{sec:sphere}
Our goal in this section is to prove \Cref{thm:Q:Ghys} which in particular solves Ghys's question (see Section~\ref{sec:intro:Ghys}) about invariants of flat $S^3$-bundles. We first describe a strategy to determine the low degree bounded cohomology of $\Diff_\circ^r(S^n)$ for $r\neq n+1$ and later we restrict to $S^3$.

\smallskip

Before proceeding, we recall that a \textbf{quasimorphism} on a group $G$ is a map $f\colon G \to \bR$ such that the quantity $\big|f(x) + f(x) - f(xy) \big|$ is bounded uniformly over $x,y\in G$. Every quasimorphism $f$ lies at bounded distance from a unique \textbf{homogeneous} quasimorphism $\bar f$, that is, a quasimorphism satisfying $\bar f (x^n) = n \bar f(x)$ for all $x\in G$ and $n\in \bZ$, see e.g.~\cite[\S3.3]{Bavard91}. A quasimorphism is called \textbf{trivial} if it is at bounded distance of a true homomorphism, or equivalently if $\bar f$ is a true homomorphism. Noting that the quantity $f(x) + f(x) - f(xy)$ is a coboundary (in the inhomogeneous model) and hence a bounded cocycle, one verifies readily that the space of quasimorphisms modulo trivial quasimorphisms is isomorphic to the kernel of the comparison map $H^2_b(G) \to H^2(G)$ from bounded to ordinary cohomology in degree two (compare again~\cite[\S3.3]{Bavard91}).

In addition, Matsumoto--Morita have shown~\cite[Cor.~2.11]{matsumoto1985bounded} that this comparison map is injective if $G$ is \textbf{uniformly perfect}, that is, if there is a bound $N$ such that every element of $G$ can be expressed as a product of at most $N$ commutators. We shall call this the ``Matsumoto--Morita lemma'' to distinguish it from their bounded acyclicity theorem. This lemma can also be checked directly on the above description in terms of quasimorphisms.

\smallskip

Back to $S^n$, Tsuboi~\cite{tsuboi2013homeomorphism, MR2509724} proved that $\Homeo_\circ(S^n)$ and $\Diff_\circ^r(S^n)$ for $r\neq n+1$  are uniformly perfect and hence Matsumoto--Morita's lemma implies that the map
 \[
 H^2_b(\Homeo_\circ(S^n))\lra H^2(\Homeo_\circ(S^n))
 \]
 is injective and the same holds for $\Diff_\circ^r(S^n)$ when $r\neq n+1$. By Thurston's theorem (\cite[Cor.~(b) of Thm.~5]{thurston1974foliations}), we know that
\[
H^2(\Homeo_\circ(S^n)) \ \cong \ H^2(\BH_\circ^\topo(S^n)).
\]
Given what is known of the right hand side, this already implies $H^2_b(\Homeo_\circ(S^n))=0$ when $n=2,3$.

As announced in \Cref{thm:sphere:intro}, we can establish this vanishing also for $H^3_b$, for all $n$, and perhaps surprisingly for diffeomorphisms as well.

\begin{proof}[Proof of \Cref{thm:sphere:intro}]
We apply \Cref{thm:sss} to the action of the group $G=\Homeo_\circ(S^n)$ or $\Diff^r_\circ(S^n)$ on the semi-simplicial set $X_{\bullet}^\perp$ of tuples of fat points in $S^n$ with disjoint cores, exactly as for \Cref{circle}. We work with $C^r$ fat point in the case of $\Diff^r_\circ(S^n)$. The bounded acyclicity of $X_{\bullet}^\perp$ is granted by \Cref{prop:generic}. The difference with the case $n=1$ is that we do not know the higher bounded cohomology of stabilisers and therefore we take $N=4$ in \Cref{thm:sss}. Thus, we need to establish the vanishing of $H^q_b(G_1)$ for all $q>0$ with $p+q<4$, where $G_1$ denotes the stabiliser of a point in $X_{p}^\perp$ and $p\geq 0$. We note that condition~\eqref{pr:ss:ub} of \Cref{thm:sss} then follows since there is only one type of stabiliser in $X_{p}^\perp$ for each $p$.

The case $p=0$ for homeomorphisms follows from the Matsumoto--Morita theorem since in that case $G_1\cong \Homeo_c(\bR^n)$. In the $C^r$ case we quote \Cref{thm:MR^n} instead.

Since $H^1_b$ vanishes for every group, the only remaining case is the vanishing of $H^2_b(G_1)$ when $p=1$. Now $G_1 \cong \Homeo_c(S^{n-1}\times \bR)$ or $\Diff^r_{c}(S^{n-1}\times \bR)$ and we can apply \Cref{thm:diff:ba}.

Now \Cref{thm:sss} shows that the bounded cohomology of $G$ coincides with that of the quotient semi-simplicial set $X_{\bullet}^\perp /G$. This quotient, however, is boundedly acyclic as soon as $n>1$ because $G$ acts transitively on tuples of fat points with distinct cores by a germ pasting argument as in the proof of \Cref{circle}. That is, given two $(p+1)$-tuples of pairwise disjoint orientation-preserving embeddings of open balls into $S^n$, there is an orientation-preserving transformation of $S^n$ in the corresponding regularity which sends each ball embedding among the first tuple to the corresponding embedding among the second tuple.  The only difference with $S^1$ for this transitivity is that in dimension one the cyclic order of the points must be preserved.
\end{proof}

Now we restrict to the case of $S^3$ and proceed to answer Ghys's question.

\begin{proof}[Proof of \Cref{thm:Q:Ghys}] The proof follows exactly the
strategy of \Cref{thm:sphere:intro}, but with $N=5$ instead of $N=4$, which requires the
vanishing of $H^2_b$ of the stabiliser of a simplex in $X_2$. But this stabiliser group is isomorphic to $\Diff^r_c(P)$ where $P$ is obtained by removing $3$ disjoint closed balls from $S^3$. The mapping class group $\pi_0(\Diff^r_c(P)^\topo)$ is generated by the Dehn twists around sphere boundary components so it is a finite $2$-torsion group (see \cite[Lemma~3.2]{hatcher1990finite}). Hence, it is enough to show that $H^2_b(\Diff^r_{c,\circ}(P))=0$. To do so, we first observe that the comparison map
\begin{equation}\label{comparison}
H^2_b(\Diff^r_{c,\circ}(P))\lra H^2(\Diff^r_{c,\circ}(P)),
\end{equation}
is injective, and then we show that its image is trivial. By ~\cite[Thm.~4.1]{fukui2019uniform} we know that $\Diff^r_{c,\circ}(P)$ is uniformly perfect for $r\neq 4, 0$. Therefore, in these regularities, Matsumoto--Morita's lemma implies that the map~\eqref{comparison}
 is injective. In particular, there is no nontrivial quasimorphism on $\Diff^r_{c,\circ}(P)$. 
 
 It is likely that $\Homeo_{c,\circ}(P)$ is also uniformly perfect but we learned from Bowden that one can use automatic continuity of homogeneous quasimorphisms on diffeomorphism groups (\cite[Thm.~A.5]{bowden2019quasi}) and the fact that homeomorphisms of $3$-manifolds can be $C^0$-approximated by diffeomorphisms (\cite[Thm.~6.3]{MR121804}) to deduce that if there were a nontrivial quasimorphisms on $\Homeo_{c,\circ}(P)$, it would restrict to a nontrivial quasimorphisms on $\Diff_{c,\circ}(P)$ which we know they do not exists by the uniform perfectness of $\Diff_{c,\circ}(P)$. Hence, we also have the injectivity for the map
\begin{equation}\label{comparison'}
  H^2_b(\Homeo_{c,\circ}(P))\lra H^2(\Homeo_{c,\circ}(P)).
 \end{equation}
Now we need the following input about the cohomology of $\Diff^r_{c,\circ}(P)^\topo$.
\begin{claim*}
$H^2(B\Diff^r_{c,\circ}(P)^\topo)=0$.
\end{claim*}
\noindent{\it Proof of the claim:} Since $\Diff^r_{c,\circ}(P)^\topo$ is a connected group, the classifying space $B\Diff^r_{c,\circ}(P)^\topo$ is simply connected. Hence, it is enough to show that $H_2(B\Diff^r_{c,\circ}(P)^\topo;\bR)=0$. The action of $\Diff^r_{c,\circ}(S^2\times \bR)^\topo$ on the space of embeddings $\text{Emb}_\circ(D^3, S^2\times \bR)$ gives rise to the fibration (see \cite{MR123338}) 
\[
\Diff^r_{c,\circ}(P)^\topo\to \Diff^r_{c,\circ}(S^2\times \bR)^\topo\xrightarrow{\mathrm{res}} \text{Emb}_\circ(D^3, S^2\times \bR),
\]
where $\mathrm{res}$ is the restriction of a diffeomorphism to a fixed embedding of $D^3$ into $S^2\times \bR$. It is standard (e.g. \cite[Thm.~9.1.2]{kupers2019lectures}) to see that $\text{Emb}_\circ(D^3, S^2\times \bR)\simeq \text{Fr}^+(S^2\times [0,1])$ where $\text{Fr}^+(-)$ means the oriented orthonormal frame bundle. Since the tangent bundle of $S^2\times [0,1]$ is trivial (which is in fact true for all orientable $3$-manifolds), we have $\text{Fr}^+(S^2\times [0,1])\cong S^2\times [0,1]\times \text{SO}(3)$. Hence, if we deloop the above fibration, we obtain a fibration
\begin{equation}\label{fib}
S^2\times [0,1]\times \text{SO}(3)^\topo\to B\Diff^r_{c,\circ}(P)^\topo\to B\Diff^r_{c,\circ}(S^2\times \bR)^\topo.
\end{equation}
Now, by Hatcher's theorem (\cite[Appendix]{hatcher1983proof}), we know that $B\Diff_{c}(S^2\times \bR)^\topo\simeq \text{SO}(3)^\topo$. Since $B\Diff_{c,\circ}(S^2\times \bR)^\topo$ is the universal cover for  $B\Diff_{c}(S^2\times \bR)^\topo$, we conclude $B\Diff_{c,\circ}(S^2\times \bR)^\topo\simeq \text{SU}(2)$. Hence, to calculate $H_2(B\Diff^r_{c,\circ}(P)^\topo;\bR)$, we look at the second page of the Serre spectral sequence for the fibration~\eqref{fib}, where we have a differential $$d_2\colon E^2_{3,0}=H_3(B\Diff_{c,\circ}(S^2\times \bR)^\topo;\bR)\to E^2_{0,2}=H_2(S^2\times [0,1]\times \text{SO}(3)^\topo;\bR),$$ whose cokernel is $H_2(B\Diff^r_{c,\circ}(P)^\topo;\bR)$.

To determine this differential, note that $B\Diff_{c,\circ}(S^2\times \bR)^\topo$ is $2$-connected, so we have $H_3(B\Diff_{c,\circ}(S^2\times \bR)^\topo;\bR)= H_2(\Diff_{c,\circ}(S^2\times \bR)^\topo;\bR)$. Therefore, the differential $d_2$ is the map induced by the map 
\[
\mathrm{res}\colon  \Diff^r_{c,\circ}(S^2\times \bR)^\topo\to S^2\times [0,1]\times \text{SO}(3)^\topo
\]
on the second homology. 

To show that $\text{res}$ induces a surjection on $H_2$, first note that $H_2(S^2\times [0,1]\times \text{SO}(3)^\topo;\bR)$ is generated by the $S^2$ factor. On the other hand, by Hatcher's theorem (\cite[Appendix]{hatcher1983proof}), we know that $ \Diff^r_{c,\circ}(S^2\times \bR)^\topo\simeq \Omega(\text{SU}(2)^\topo)$ where $\Omega(\text{SU}(2)^\topo)$ is the loop space on $\text{SU}(2)^\topo$. Hence, the map $d_2$ is induced by the natural map $$\alpha\colon\Omega(\text{SU}(2)^\topo)\to S^2\times [0,1],$$ which is given by the Hopf map to the slice $S^2\times\{ t\}$ at time $t$ of the loop. It can be easily seen that $\alpha$ induces an isomorphism on $H_2$. So $d_2$ is also an isomorphism. Therefore, we have $H_2(B\Diff^r_{c,\circ}(P)^\topo;\bR)=0$. $\blacksquare$

We use this topological fact as an input to show that the image of the comparison map for $\Diff^r_{c,\circ}(P)$ is trivial. This is easier for $\Homeo_{c,\circ}(P)$ since by the classical theorem of Cerf (\cite{MR116351}), the proof of Smale's conjecture (\cite[Appendix]{hatcher1983proof}) implies that $\Homeo_{c,\circ}(P)^\topo\simeq \Diff^r_{c,\circ}(P)^\topo$. Hence, by Thurston's theorem (\cite[Cor.~(b) of Thm.~5]{thurston1974foliations}), we have 
\[
H^2(\Homeo_{c,\circ}(P))=H^2(B\Homeo_{c,\circ}(P)^\topo)=0.
\]
Therefore, the injectivity of the comparison map~\eqref{comparison'} already implies that $H^2_b(\Homeo_{c,\circ}(P))=0$. For the case of $\Diff^r_{c,\circ}(P)$, we do not know whether $H^2(\Diff_{c,\circ}(P))$ vanishes so we need to work a little harder.
\begin{claim*}Let $M^n$ be a $n$-manifold such that $H^2(B\Diff^r_{c,\circ}(M)^\topo)=0$ and suppose that we know that the comparison map~\eqref{comparison} is injective for $\Diff^r_{c,\circ}(M)$. Then for $r\neq n+1$ we have $H^2_b(\Diff^r_{c,\circ}(M))=0$. 
\end{claim*}
\noindent{\it Proof of the claim:} Recall that $\overline{B\Diff^r_{c,\circ}(M)}$ is homotopy fiber in the fibration 
\[
\overline{B\Diff^r_{c,\circ}(M)}\to B\Diff^r_{c,\circ}(M)\to B\Diff^r_{c,\circ}(M)^\topo.
\]
Given that $B\Diff^r_{c,\circ}(M)^\topo$ is simply connected, and we know that for $r\neq \text{dim}(M)+1$, we have $H^1(\overline{B\Diff^r_{c,\circ}(M)})=0$ by \cite[Appendix]{MR743944} and also we have $H^2(B\Diff^r_{c,\circ}(M)^\topo)=0$ by the hypothesis, the spectral sequence for the above fibration implies that the map 
\[
H^2(\Diff^r_{c,\circ}(M))\to H^2(\overline{B\Diff^r_{c,\circ}(M)}),
\]
is injective. On the other hand, from Mather--Thurston's theorem for $r\neq \text{dim}(M)+1$, it follows that  $H^2(\overline{B\Diff^r_{c,\circ}(M)})$ only depends on $\text{dim}(M)$ (see \cite[Second corollary at page 306]{thurston1974foliations} and the proof of the lemma in \cite[Appendix]{MR743944}). Therefore, the inclusion $\Diff_{c,\circ}^r(\bR^n)\to \Diff^r_{c,\circ}(M)$ induces an isomorphism
\[
H^2(\overline{B\Diff^r_{c,\circ}(M)})\xrightarrow{\cong} H^2(\overline{B\Diff^r_{c,\circ}(\bR^n)}).
\]
It follows from the commutative diagram 
\[
 \begin{tikzpicture}[node distance=1.5cm, auto]
  \node (A) {$ H^2(\Diff^r_{c,\circ}(M))$};
  \node (B) [right of=A, node distance=4cm] {$H^2(\overline{B\Diff^r_{c,\circ}(M)})$};
  \node (C) [below of=A]{$H^2(\Diff^r_{c,\circ}(\bR^n))$};
  \node (D) [right of=C, node distance=4cm]{$H^2(\overline{B\Diff^r_{c,\circ}(\bR^n)}),$};
  \draw [right hook ->] (A) to node {$$}(B);
    \draw [->] (A) to node {$$}(C);
  \draw [->] (C) to node {$$}(D);
  \draw [->] (B) to node {$\cong$}(D);
\end{tikzpicture}
\]
that the left vertical map is also injective. Now the commutative diagram of comparison maps
\[
 \begin{tikzpicture}[node distance=1.5cm, auto]
  \node (A) {$ H_b^2(\Diff^r_{c,\circ}(M))$};
  \node (B) [right of=A, node distance=4cm] {$H^2(\Diff^r_{c,\circ}(M))$};
  \node (C) [below of=A]{$H_b^2(\Diff^r_{c,\circ}(\bR^n))$};
  \node (D) [right of=C, node distance=4cm]{$H^2(\Diff^r_{c,\circ}(\bR^n)),$};
  \draw [right hook ->] (A) to node {$$}(B);
    \draw [->] (A) to node {$$}(C);
  \draw [->] (C) to node {$$}(D);
  \draw [right hook->] (B) to node {$$}(D);
\end{tikzpicture}
\]
implies that the left vertical map is also injective. 

But we know that $H_b^2(\Diff^r_{c,\circ}(\bR^n))$ is trivial by \Cref{thm:diff:ba}. Therefore, $H_b^2(\Diff^r_{c,\circ}(M))$ also vanishes. $\blacksquare$

Hence, by the claim, for $r\neq 4$, we obtain $H^2_b(\Diff_{c,\circ}(P))=0$ which finishes the proof.
\end{proof}

\section{Further comments and questions}
\subsection{Homeomorphisms of certain geometric 3-manifolds}\label{3fold}
Burago--Ivanov--Polterovich (\cite[Section 3.3]{burago2008conjugation}) proved that $\Diff_{\circ}(M)$ is uniformly perfect for any closed $3$-manifold $M$. Therefore, as we have seen already by Matsumoto--Morita's lemma, the comparison map
 \[
 H^2_b(\Diff_\circ(M))\to H^2(\Diff_\circ(M))
 \]
 is injective. Again by automatic continuity of homogeneous quasimorphisms on diffeomorphism groups (\cite[Thm.~A.5]{bowden2019quasi}) and the fact that homeomorphisms of $3$-manifolds can be $C^0$-approximated by diffeomorphisms (\cite[Thm.~6.3]{MR121804}), we can also conclude that 
 \[
 H^2_b(\Homeo_\circ(M))\to H^2(\Homeo_\circ(M))
 \]
 is injective. By Thurston's theorem (\cite[Cor.~(b) of Thm.~5]{thurston1974foliations}), we know that $$H^2(\Homeo_\circ(M))=H^2(\BH_\circ^\topo (M)).$$ Thanks to the generalised Smale conjecture which has been extensively studied for many cases (\cite{hatcher1983proof, MR0420620, MR0448370, gabai2001smale, MR2976322, MR3024309}) and recently has been proved in the remaining cases by Bamler and Kleiner (\cite{bamler2019ricci, bamler2019diffeomorphisms}), we know the homotopy type of $\Homeo_\circ^\topo(M)$ for a $3$-manifold admitting a Thurston geometry. Combining a number of known facts about $\Homeo_\circ(M)$, we prove:
 
 \begin{thm}
   Let $M$ be a closed hyperbolic $3$-manifold or a closed Seifert fibered space whose fundamental group has $\bZ$ as its center.

   Then $H^2_b(\Homeo_\circ(M))=0$.
\end{thm}

 \begin{proof} In the hyperbolic case, the generalised Smale conjecture was proved by Gabai (\cite{gabai2001smale}) which implies that $\Homeo_\circ^\topo (M)\simeq *$ as a topological group. Therefore, by the above argument we obtain $H^2_b(\Homeo_\circ(M))=0$. However for Seifert fibered manifold $M$, the generalised Smale conjecture implies that $\Homeo_\circ^\topo (M)\simeq (S^1)^k$ where $k$ is the rank of the center of $\pi_1(M)$ except the case of the solid torus for which $\Homeo_\circ^\topo (M)\simeq S^1\times S^1$ and the case of $D^3$ for which $\Homeo_\circ^\topo (D^3)\simeq \mathrm{SO}(3)^\topo$ (see the introduction of \cite{MR3024309}). Hence, in our case of interest when $H^2(\BH_\circ(M))$ is nontrivial it is isomorphic $\bR$ which is generated by the Euler class induced by the circle action on the Seifert fibered space by rotating the circle fibers. K.~Mann~\cite{mann2020unboundedness} showed that this Euler class is not a bounded class. Therefore, injectivity of $ H^2_b(\Homeo_\circ(M))\to H^2(\Homeo_\circ(M))$ implies that $H^2_b(\Homeo_\circ(M))$ is also zero in this case. 
 \end{proof}

\subsection{Diffeomorphisms and homeomorphisms of  \texorpdfstring{$\bR^n$}{Rn}}
The algebraic properties of the automorphism groups of non-compact manifolds are more subtle to study. It would be interesting to apply the same techniques we used for spheres and the $2$-disc to extract information about the bounded cohomology of diffeomorphisms and homeomorphisms of $\bR^n$. The main motivation is to study which invariants of $C^r$-flat $\bR^n$-bundles for $r\geq 0$ are bounded. For the $C^0$-case, recall that $B\Homeo_\circ(\bR^n)^\topo$ classifies oriented $\bR^n$-microbundle and it has a nontrivial homotopy type (see \cite{kupers2020diffeomorphisms} for nontrivial characteristic classes of $\bR^n$-microbundles). In particular, the Euler class $\Euler\in H^{n}(B\Homeo_\circ(\bR^{n})^\topo)$ when $n$ is even and all the Pontryagin classes $\Pont_i$ for $i\leq n/4$ are nontrivial. Hence, they also pull back nontrivially to $H^{\bullet}(\Homeo_\circ(\bR^{n}))$ since by McDuff's theorem (\cite{mcduff1980homology}) we know that $$H^{\bullet}(B\Homeo_\circ(\bR^{n})^\topo)\cong H^{\bullet}(\Homeo_\circ(\bR^{n})).$$ In fact, these classes are also nontrivial for $C^r$-flat $\bR^{n}$-bundles for all regularities $r>0$ by the following observation.   Using a deep result of Segal (\cite[Prop.~1.3 and~3.1]{segal1978classifying}), we know that there is a map $$\BdrDiff(\bR^n)\to \mathrm{B}\Gamma^{r}_{n,\circ},$$ which is a homology isomorphism, where $\mathrm{B}\Gamma^{r}_{n, \circ}$ is the classifying space of Haefliger structures for codimension $n$ foliations that are transversely oriented,  see Section~1 in~\cite{segal1978classifying} for more details. On the other hand, there is map 
\[
\nu\colon \mathrm{B}\Gamma^{r}_{n,\circ}\to B\text{GL}_n(\bR)^\topo_{\circ},
\]
which classifies oriented normal bundles to the codimension $n$ foliations. For all regularities, it is known that the map $\nu$ is at least $(n+1)$-connected, see Remark~1, Section~II.6 in~\cite{haefliger1971homotopy}. Hence, in particular the induced map 
\[
H^\bullet(B\text{GL}_n(\bR)^\topo_{\circ})\to H^\bullet(\Diff^r_\circ(\bR^n)),
\]
is an isomorphism for $\bullet\leq n$. Therefore, the classes $\Pont_i$ for $i\leq n/4$  are nontrivial and so is $\Euler$ when $n$ is even in $H^\bullet(\Diff^r_\circ(\bR^n))$ for all $r$. However, for $n=2$, Calegari (\cite[Thm.~C]{calegari2004circular}) showed that $\Euler\in H^2(\Diff^r_\circ(\bR^2))$ is not a bounded class. We in fact show that there is no bounded invariant for flat $\bR^n$-bundles over surfaces.

 \begin{thm}\label{thm:highreg}
We have $H^2_b(\Diff^{r}_\circ (\bR^n))=0$ for all $n$ and all $r \neq n+1$.
 \end{thm}
 
\begin{proof}
 As a consequence of Rybicki's theorem (\cite[Thm.~1.2]{MR2780744}) we know that $\Diff^{r}_\circ (\bR^n)$ is uniformly perfect for $r\neq n+1$. Therefore, Matsumoto--Morita's lemma (\cite[Cor.~2.11]{matsumoto1985bounded}) implies that the map
 \[
 H^2_b(\Diff^{r}_\circ (\bR^n))\lra H^2(\Diff^{r}_\circ (\bR^n))
 \]
 is injective for $r\neq n+1$. The remaining part of the proof consists in showing that the right hand side vanishes when $n\neq 2$, while for $n=2$ it turns out that the non-zero elements of the right hand side are not in the image of the comparison map.

 Recall the deep result of Segal (\cite[Prop.~1.3 and~3.1]{segal1978classifying}) implies that there is a map $\BdrDiff(\bR^n)\to \mathrm{B}\Gamma^{r}_{n,\circ}$ which is a homology isomorphism, where $\mathrm{B}\Gamma^{r}_{n, \circ}$ is the classifying space of Haefliger structures for codimension $n$ foliations that are transversely oriented,  see Section~1 in~\cite{segal1978classifying} for more details. Since $r\neq n+1$, a theorem of Mather~\cite[Section 7]{MR0356129}, implies that the natural map $\nu\colon\mathrm{B}\Gamma^{r}_{n, \circ}\to \mathrm{B}\mathrm{GL}^{\topo}_n(\bR)_\circ$ is at least $(n+2)$-connected. (Without the restriction $r\neq n+1$, we still know that the map $\nu$ is at least $(n+1)$-connected, see Remark~1, Section~II.6 in~\cite{haefliger1971homotopy}.) Therefore, by combining Segal's theorem and Thurston's theorem, we have $H^2(\Diff^{r}_\circ (\bR^1))=0 $ for $r\neq 2$ as desired for $n=1$. For $n=2$, we have $H^2(\Diff^{r}_\circ (\bR^2))=\bR$ generated by the Euler class for all $r$. On the other hand, a theorem of Calegari (\cite[Thm.~C]{calegari2004circular}) shows that the Euler class in $H^2(\Diff^{r}_\circ (\bR^2))=\bR$  is not a bounded class. Finally, for $n>2$, since $H^2(\mathrm{B}\mathrm{GL}^{\topo}_n(\bR)_\circ;\bR)=0$, we have indeed $H^2(\Diff^{r}_\circ (\bR^n))=0 $.
 \end{proof}

\subsection{Questions}
We shall end with few questions regarding the borderline of some of the cases we considered. For the case of spheres, it would be interesting to see if $H^4_b$ vanishes for $S^n$ where $n\neq 1,3$. In particular, the case of $S^2$ is already interesting.   By Thurston's theorem (\cite[Cor.~(b) of Thm.~5]{thurston1974foliations}) we know $$H^4(\Homeo_\circ(S^2))=H^4(\BH_\circ^\topo(S^2)),$$ and by Hamstrom's theorem  (\cite{hamstrom1974homotopy}) we know that $\Homeo_\circ^\topo(S^2)\simeq \mathrm{SO}(3)^\topo$. Therefore, we have $H^4(\Homeo_\circ(S^2))=\bR$ generated by the Pontryagin class $\Pont_1$ in $H^4(\mathrm{B}\mathrm{SO}(3)^\topo)$. 
 \begin{quest}
  Is the first Pontryagin class bounded or even more generally is $H^4_b(\Homeo_\circ(S^2))$ nontrivial?
 \end{quest}
 \begin{quest}
 For $n>1$, is there any nontrivial bounded class in $H^\bullet(\Homeo_\circ(S^n))$?
 \end{quest}
 In proving $\Homeo_\circ(D^2)\to \Homeo_\circ(S^1)$ we used a huge semi-simplicial set of transverse fat chords. There is an obvious way to generalize this semi-simplicial set to higher dimensions which is still boundedly acyclic. But then, given that the complement of simplices in $D^n$ could be more complicated when $n>2$, it is not clear whether the stabilisers are boundedly acyclic. One could try to define ``smaller" semi-simplicial set whose stabilisers of its simplices are among the bounded acyclic groups. But it becomes harder to prove that the chosen semisimplicial set is boundedly acyclic. Hence, we pose the generalisation of our result about $\Homeo_\circ(D^2)$ as a question.
 \begin{quest}
 Does the restriction map $\Homeo_\circ(D^n)\to \Homeo_\circ(S^{n-1})$ induce an isomorphism on bounded cohomology for $n>2$?
 \end{quest}
 We generalised Calegari's theorem (\cite[Thm.~C]{calegari2004circular}) about the second bounded cohomology of $\Diff^{r}_\circ (\bR^2)$ by showing the vanishing $H^2_b(\Diff^{r}_\circ (\bR^n))=0$ for all $n$ and all $r \neq n+1$.  So the higher degrees of bounded cohomology of $\Diff^{r}_\circ (\bR^2)$ remains to be determined.
 \begin{quest}
 Is $\Diff^r_\circ(\bR^n)$ a boundedly acyclic group? Is $\Euler\in H^{2n}(\Diff^r_\circ(\bR^{2n}))$ a bounded class?
 \end{quest}
And finally, as we mentioned in the introduction, our proof of the unboundedness of the Euler class for oriented $C^0$-flat $S^3$-bundles is not constructive. It would be geometrically enlightening to find a constructive proof.
\begin{quest}
Find explicit families of oriented $C^0$-flat $S^3$ bundles over a given $4$-manifold with unbounded Euler number.
\end{quest}

\bibliographystyle{alpha}
\bibliography{reference}
\end{document}